\documentclass[11pt]{article}

\usepackage[hidelinks]{hyperref}
\hypersetup{
  colorlinks   = true, 
  urlcolor     = blue, 
  linkcolor    = blue, 
  citecolor   = red 
}
\usepackage{amsmath,amsthm,amssymb}
\usepackage{graphicx}
\usepackage{bbm}
\newcommand{\sy}{\boldsymbol{\Psi}}
\newcommand{\py}{\boldsymbol{\Phi}}
\newcommand{\T}{\mathbb{T}^N}

\usepackage{xcolor}
\usepackage[margin=1in]{geometry} 
\usepackage{enumerate,mathtools,mathrsfs,bm,graphicx}
\usepackage[numbers, super]{natbib}
\usepackage[hidelinks]{hyperref}
\newcommand{\N}{\mathbb{N}}									

\newcommand{\R}{\mathbb{R}}

\newcommand{\vertiii}[1]{{\left\vert\kern-0.25ex\left\vert\kern-0.25ex\left\vert #1 
    \right\vert\kern-0.25ex\right\vert\kern-0.25ex\right\vert}}
    
\newcommand{\overbar}[1]{\mkern 1.5mu\overline{\mkern-1.5mu#1\mkern-1.5mu}\mkern 1.5mu}
\newcommand{\inner}[2]{\left\langle #1, #2 \right\rangle}
\newcommand{\floor}[1]{\left \lfloor{#1}\right \rfloor}		
\DeclarePairedDelimiter\abs{\lvert}{\rvert}					
\newcommand{\norm}[1]{\left\Vert #1 \right\Vert}

\newtheorem{theorem}{Theorem}[section]
\newtheorem{corollary}{Corollary}[theorem]

\newtheorem{lemma}[theorem]{Lemma}
\newtheorem{proposition}[theorem]{Proposition}
\newtheorem{remark}[theorem]{Remark}
\newtheorem{definition}[theorem]{Definition}

\newtheorem{assumption}[theorem]{Assumption}

\begin{document}
	\title{High Order Smoothness for Stochastic Navier-Stokes Equations with Transport and Stretching Noise on Bounded Domains}
	\author{Daniel Goodair\footnote{EPFL, 
    daniel.goodair@epfl.ch. Supported
    by the EPSRC Project 2478902.}}
	\date{\today} 
	\maketitle
\setcitestyle{numbers}	
\thispagestyle{empty}
\begin{abstract}
    We obtain energy estimates for a transport and stretching noise under Leray Projection on a 2D bounded convex domain, in Sobolev Spaces of arbitrarily high order. The estimates are taken in equivalent inner products, defined through powers of the Stokes Operator with a specific choice of Navier boundary conditions. We exploit fine properties of the noise in relation to the Stokes Operator to achieve cancellation of derivatives in the presence of the Leray Projector. As a result, we achieve an additional degree of regularity in the corresponding Stochastic Navier-Stokes Equation to attain a true strong solution of the original Stratonovich equation. Furthermore for any order of smoothness, we can construct a strong solution of a hyperdissipative version of the Stochastic Navier-Stokes Equation with the given regularity; hyperdissipation is only required to control the nonlinear term in the presence of a boundary. We supplement the result by obtaining smoothness without hyperdissipation on the torus, in 2D and 3D on the lifetime of solutions. 
\end{abstract}
	
\tableofcontents
\textcolor{white}{Hello}
\thispagestyle{empty}
\newpage

\setcounter{page}{1}
\addcontentsline{toc}{section}{Introduction}

\section*{Introduction}

We are concerned with the high order smoothness of solutions to the 2D incompressible Navier-Stokes equation under Stochastic Advection by Lie Transport (SALT) introduced in [\cite{holm2015variational}], given by
\begin{equation} \label{number2equationSALT}
    u_t - u_0 + \int_0^t\mathcal{L}_{u_s}u_s\,ds - \nu\int_0^t \Delta u_s\, ds + \int_0^t B(u_s) \circ d\mathcal{W}_s + \nabla \rho_t= 0
\end{equation}
where $u$ represents the fluid velocity, $\rho$ the pressure\footnote{The pressure term is a semimartingale, and an explicit form for the SALT Euler Equation is given in [\cite{street2021semi}] Subsection 3.3}, $\mathcal{W}$ is a Cylindrical Brownian Motion, $\mathcal{L}$ represents the nonlinear term and $B$ is a first order differential operator (the SALT Operator) properly introduced in Subsection \ref{sub salt}. This injection of noise into the system follows a variational principle as presented in [\cite{holm2015variational}], adding uncertainty in the transport of fluid parcels to reflect the unresolved scales. Indeed, the physical significance of such transport and advection noise in modelling, numerical analysis and data assimilation continues to be well documented, see [\cite{alonso2020modelling}, \cite{cotter2020data}, \cite{cotter2018modelling}, \cite{cotter2019numerically}, \cite{crisan2021theoretical}, \cite{dufee2022stochastic}, \cite{flandoli20212d}, \cite{flandoli2022additive}, \cite{holm2020stochastic}, \cite{holm2019stochastic}, \cite{lang2022pathwise}, \cite{van2021bayesian}, \cite{street2021semi}] as well as the recent book [\cite{flandoli2023stochastic}]. Intrinsic to this stochastic methodology is that $B$ is defined relative to a collection of functions $(\xi_i)$ which physically represent spatial correlations. These $(\xi_i)$ can be determined at coarse-grain resolutions from finely resolved numerical simulations, and mathematically are derived as eigenvectors of a velocity-velocity correlation matrix (see [\cite{cotter2020data}, \cite{cotter2018modelling}, \cite{cotter2019numerically}]).\\

The existence of strong solutions to (\ref{number2equationSALT}) on the torus is well understood, see [\cite{agresti2024stochastic}, 
  \cite{alonso2020well}, \cite{crisan2019solution},  \cite{crisan2021theoretical}, \cite{crisan2023well}, \cite{crisan2022spatial}, \cite{lang2023well}, \cite{lang2023analytical}, \cite{flandoli2021delayed}, \cite{tang2023stochastic}] for the stochastic Navier-Stokes equation and related models with transport and advection noise. On a bounded domain, the situation is drastically different. Prior to the author's work [\cite{goodair2023navier}, \cite{goodair20233d}], we are only aware of one existence result for analytically strong solutions to a fluid equation perturbed by a transport type noise. This was given in [\cite{brzezniak2021well}] where the authors assume that the gradient dependency is small relative to the viscosity, which is necessary in the It\^{o} case but avoids the technical complications of demanding a cancellation of derivatives in the presence of the boundary. For a Stratonovich transport noise such an assumption may not be necessary, as conversion to It\^{o} form yields a corrector term which may provide this cancellation of the top order derivative arising from the noise in energy estimates. This type of estimate has been well understood since the works [\cite{gyongy1989approximation}, \cite{gyongy2003splitting}, \cite{gyongy1992stochastic}] and was demonstrated for SALT noise in the presence of a boundary in [\cite{goodair20233d}]. The issue on a bounded domain is that such a control does not practically transfer to energy estimates of the stochastic fluid equation. As is classical, we work with a projected form of the equation \begin{equation} \label{projected strato Salt}
    u_t = u_0 - \int_0^t\mathcal{P}\mathcal{L}_{u_s}u_s\ ds - \nu\int_0^t A u_s\, ds  - \int_0^t \mathcal{P}Bu_s \circ d\mathcal{W}_s 
\end{equation}
to relieve ourselves of the pressure, where $\mathcal{P}$ is the Leray Projector onto the space of divergence-free functions with zero normal boundary component, and $A = -\mathcal{P}\Delta$. These operators and their properties are more thoroughly introduced in Subsection \ref{functional framework subsection}. Therefore one must obtain sufficient estimates on the noise under Leray Projection, inconsequential in $L^2$ however significant in higher order spaces as the Leray Projector is no longer self-adjoint. This elucidates the difficulty in recovering analytically strong solutions as opposed to weak solutions, as the energy norm of weak solutions arises from $L^2$ estimates. For existence and uniqueness results of weak solutions in two and three dimensions on a bounded domain, we refer the reader to [\cite{goodair2023navier}, \cite{goodair2023zero}, \cite{mikulevicius2005global}]. Moreover the contrast between torus and bounded domain is indicated, as on the torus the Leray Projector commutes with derivatives so is in fact self-adjoint in high order Sobolev Spaces. Its presence, therefore, does not affect the traditional estimates and cancellation for transport noise in such spaces.\\

In fact, the problem is deeper for the Navier-Stokes equation and viscous models more generally. Precision is required to extract the gain in regularity from the viscous term, restricting our energy estimates to be taken in inner products of powers of the Stokes Operator arising from its spectral theory. In many cases there is a subtlety to this. The equivalent $W^{1,2}$ inner product, begetting strong solutions, is simply the homogeneous counterpart both on the torus and for the classical no-slip boundary condition (where $u = 0$ on the boundary $\partial \mathscr{O}$). On the torus, the equivalence is true even for higher order Sobolev Spaces: one can see [\cite{robinson2016three}] Section 2.3 for these facts. The equivalence however, and more so our ability to use these inner products at all, is only valid in the right domain. For the torus this is always the divergence-free and zero-mean subspace, which is mapped to in Leray Projection so is automatically valid. On a bounded domain we require the divergence-free subspace with some boundary condition, one that must often be significantly more restrictive than the actual boundary condition required of the solution. For example with the no-slip boundary, the best sufficient condition that we have for the corresponding $W^{m,2}$ subspace is for the function to belong to $W^{m-1,2}_0$: far stronger than the $W^{1,2}_0$ requirement of the solution.\\

Even in the base regularity of strong solutions in $W^{1,2}$, this is a significant and presently insurmountable problem for the no-slip boundary condition. Transport noise aside, simply for a general Lipschitz multiplicative noise the existence of analytically strong solutions is unknown. This case is often considered resolved due to the work of Glatt-Holtz and Ziane in [\cite{glatt2009strong}], where indeed the existence of analytically strong solutions for a Lipschitz multiplicative noise is shown, however this noise is taken as a mapping from $W^{1,2}_{\sigma}$ (the subspace of divergence-free, zero-trace functions) into itself. If we were to take $B$ in (\ref{number2equationSALT}) as mapping $W^{1,2}$ into even compactly supported functions, we would still \textit{not} expect $\mathcal{P}B$ to map from $W^{1,2}_{\sigma}$ into itself: this is due to the fact that the Leray Projector destroys the zero-trace property in general. Thus, this assumption of Glatt-Holtz and Ziane is far more restrictive than a bounded noise at the level of the true equation (\ref{number2equationSALT}). Moreover we could remove the transport term in our SALT noise, leaving only a bounded stretching term, and this would not satisfy the assumptions in [\cite{glatt2009strong}]. The issue of noise not retaining the central $W^{1,2}_{\sigma}$ space has also been recognised in [\cite{agresti2024critical}] Remark 5.6, preventing strong solutions from being obtained in the variational framework.\\

Fortunately, a change to Navier boundary conditions facilitated the existence of analytically strong solutions in the author's paper [\cite{goodair2023navier}]; these are
\begin{equation} \label{navier boundary conditions}
    u \cdot \mathbf{n} = 0, \qquad 2(Du)\mathbf{n} \cdot \mathbf{\iota} + \alpha u\cdot \mathbf{\iota} = 0
\end{equation} 
where $\mathbf{n}$ is the unit outwards normal vector, $\mathbf{\iota}$ the unit tangent vector, $Du$ is the rate of strain tensor $(Du)^{k,l}:=\frac{1}{2}\left(\partial_ku^l + \partial_lu^k\right)$ and $\alpha \in C^2(\partial \mathscr{O};\R)$ represents a friction coefficient which determines the extent to which the fluid slips on the boundary relative to the tangential stress. These conditions were first proposed by Navier in [\cite{navier1822memoire}, \cite{navier1827lois}], and have been derived in [\cite{maxwell1879vii}] from the kinetic theory of gases and in [\cite{masmoudi2003boltzmann}] as a hydrodynamic limit. Furthermore they have proven viable for modelling rough boundaries as seen in [\cite{basson2008wall}, \cite{gerard2010relevance}, \cite{pare1992existence}]. The derivatives appearing in (\ref{navier boundary conditions}) mean that the space of functions satisfying Navier boundary conditions is not closed in $W^{1,2}$, such that the $W^{1,2}$ space induced from the spectral theory of the Stokes Operator does not necessitate the full Navier boundary conditions. This space, which we call $\bar{W}^{1,2}_{\sigma}$, is simply the subspace of divergence-free functions with zero normal boundary component. Furthermore the Leray Projector maps into this space, overcoming the issues imposed by the no-slip condition.\\

From the described picture, one can see the difficulty of contemplating higher order smoothness. As soon as we pass up to $W^{2,2}$, the full condition (\ref{navier boundary conditions}) would be necessary. As described in the no-slip setting, higher regularity enforces even more stringent boundary requirements. To solve this, we take a specific choice of the Navier boundary conditions owing to their relation with vorticity. Indeed if $w$ is the vorticity of the fluid and $\kappa \in C^2(\partial \mathscr{O};\R)$ denotes the curvature of the boundary then the conditions (\ref{navier boundary conditions}) are equivalent to \begin{equation} \label{its a new rep}
    u \cdot \mathbf{n} = 0, \qquad w = (2\kappa - \alpha)u \cdot \iota.
\end{equation}
The choice $\alpha = 2\kappa$, known as the `free boundary', is particularly appealing. The condition has been central to inviscid limit theory, stemming from [\cite{lions1969quelques}] in the deterministic setting and shown by the author in [\cite{goodair2023navier}] for this stochastic system. In the present work we exploit this characterisation of the boundary condition in terms of the curl, coupled with the property that the curl passes through the Leray Projector to make the boundary value of projected functions truly tractable. In particular, whilst choice of the driving spatial correlation $\xi_i$ to approach zero quickly at the boundary does not ensure that the corresponding noise $\mathcal{P}B_iu$ is zero at the boundary, it does ensure that $\textnormal{curl}\left(\mathcal{P}B_iu \right)$ is zero at the boundary hence satisfying Navier boundary conditions for $\alpha = 2\kappa$. For higher orders it becomes sufficient to prove that $\textnormal{curl}\left(A^k\mathcal{P}B_iu\right)$ is zero on the boundary, which follows similarly.\\

Of course we have only provided the idea as to why the noise is simply valued in the right space; obtaining the necessary cancellation of derivatives in energy estimates remains an obstacle. In essence we are tasked to control a term of the form $$ \inner{A^k\mathcal{P}B_i^2u}{A^ku} + \norm{A^{k}\mathcal{P}B_iu}^2.$$
These terms arise in energy estimates from the It\^{o}-Stratonovich corrector and quadratic variation of the stochastic integral respectively. Even for the torus, we believe that such an estimate is new. We are only aware of high order estimates for transport type noise in specific cases; firstly for the vorticity form, for example in [\cite{lang2023well}], where the Leray Projector is not present and hence the usual $W^{k,2}$ inner product can be used, and secondly where very strong and unexpected conditions for the full commutativity between $A$ and $B_i$ are assumed, for example in [\cite{crisan2022spatial}] (3.6). The evident concern is that presence of the Leray Projector blocks our ability to cancel derivatives with the densely defined adjoint $B_i^*$, containing a negative of the transport term. Our approach relies strongly on the structure of the transport and advection term $B_i$ as preserving gradients, facilitating the property that $\mathcal{P}B_i=\mathcal{P}B_i\mathcal{P}$, and iterating commutator estimates of $B_i$ and $\Delta$. Use of the commutator estimates does not provide complete cancellation of the derivative, only `half' of it, though the remaining derivative dependency can be isolated to be with an arbitrarily small constant. In particular, one can use the viscous term to control it and close the estimates.\\

As a result of these estimates we obtain, for the first time with a boundary, the existence of a strong solution to the true Stratonovich equation (\ref{projected strato Salt}). Despite the profusion of recent works around Stratonovich transport type noise, results have been almost exclusively for a corresponding It\^{o} form of the equation with only a conceptual understanding of their equivalence. Due to the unbounded nature of the noise, making this equivalence rigorous comes at the cost of an additional derivative; therefore to actually obtain a solution of the Stratonovich form, the form which is physically motivated, one must show an extra degree of regularity in the It\^{o} form. The rigorous conversion was demonstrated in [\cite{goodair2022stochastic}], see Theorem \ref{theorem for ito strat conversion} in the appendix. Therefore, pathwise regularity in $C\left([0,T];W^{2,2}\right) \cap L^2\left([0,T];W^{3,2}\right)$ is required to solve the Stratonovich form in $L^2$, which is the content of Theorem \ref{Stratotheorem2}. For this and our other results, the technical restriction that the domain must be convex (that is, the curvature $\kappa$ is non-negative at all points of the boundary) is required. This owes to the fact for $j$ odd, the inner product $\inner{A^{\frac{j}{2}}\cdot}{A^{\frac{j}{2}}\cdot}$ is only shown to be equivalent to usual $W^{j,2}$ inner product for $\alpha \geq \kappa$, but as we choose $\alpha = 2\kappa$ then we need $\kappa \geq 0$.\\

Our higher order noise estimates may suggest that we can construct solutions of (\ref{projected strato Salt}) in any $W^{k,2}$ space provided the initial condition is equally regular. We have not quite achieved this, suffering from an insufficient control on the Galerkin Projection of the nonlinear term. Indeed if the Galerkin Projections were uniformly bounded in $W^{k,2}$ on the range of the nonlinear term then we would have success, though we draw attention to the fact that these projections are not uniformly bounded in general; this has certainly been missed in some texts. A counterexample is given in the appendix, Lemma \ref{lemma for explosion of galerkin}. Furthermore, it seems to the author that the expected high order energy estimates for the deterministic system are actually not known on a bounded domain. We appreciate that interior regularity estimates are classical, see for example [\cite{fujita1964navier}, \cite{heywood1980navier}]. As an application of the high order noise estimates we instead construct regular solutions of a \textit{hyperdissipative} equation,
\begin{equation} \label{number2equationSALT hyper}
    u_t - u_0 + \int_0^t\mathcal{L}_{u_s}u_s\,ds + \nu\int_0^t (-\Delta)^{\beta} u_s\, ds + \int_0^t B(u_s) \circ d\mathcal{W}_s + \nabla \rho_t= 0
\end{equation}
where $\beta > 1$, $\beta \in \N$, or more precisely its projected version 
\begin{equation} \label{projected strato Salt hyperdiss}
  u_t = u_0 - \int_0^t\mathcal{P}\mathcal{L}_{u_s}u_s\ ds - \nu\int_0^t A^\beta u_s\, ds  - \int_0^t \mathcal{P}Bu_s \circ d\mathcal{W}_s
\end{equation}
where the property that $\mathcal{P}(-\Delta)^\beta = A^{\beta}$ is discussed in Subsection \ref{functional framework subsection}. It is stressed once more that hyperdissipation is only required to control the nonlinear term under Galerkin Projection, and not to control the noise. We illustrate this by showing that solutions of (\ref{projected strato Salt}) on the 2D or 3D torus retain the smoothness of the initial condition, which still appears relevant as the the estimates on the noise are new even on the torus. For the bounded domain the size of $\beta$ will be dependent on the desired regularity. In summary, our main results are:
\begin{enumerate}
    \item Let $\mathscr{O}$ be a 2D smooth bounded convex domain, with given $u_0$ in $W^{2,2}$ satisfying Navier boundary conditions (\ref{navier boundary conditions}) for $\alpha = 2\kappa$. Then there exists a process $u$ which belongs pathwise to $C\left([0,T];W^{2,2}\right) \cap L^2\left([0,T];W^{3,2}\right)$ and satisfies (\ref{projected strato Salt}) in $L^2$; see Theorem \ref{Stratotheorem2} for the complete statement.

    \item \label{labelous} Let $\mathscr{O}$ be a 2D smooth bounded convex domain, with given $u_0$ in $W^{m,2}$ for $m \in \N$, $m \geq 2$, satisfying Navier boundary conditions (\ref{navier boundary conditions}) for $\alpha = 2\kappa$ and further belonging to the domain of $A^{\frac{m}{2}}$. Then there exists a process $u$ which belongs pathwise to $C\left([0,T];W^{m,2}\right) \cap L^2\left([0,T];W^{2m-1,2}\right)$ and satisfies (\ref{projected strato Salt hyperdiss}) for $\beta = m-1$ in $L^2$; see Theorem \ref{Stratotheorem1} for the complete statement.
\end{enumerate}

Of course the first result is a special case of the second. We detail the plan of the paper now.

\begin{itemize}
    \item Section \ref{section preliminaries} introduces some elementary notation as well as classical results in the study of Navier-Stokes equations with Navier boundary conditions. The diffusion operator $B$ is properly introduced with its essential properties presented.

    \item Section \ref{section high order estimates} proves the high order noise estimates in spaces defined through the spectrum of the Stokes Operator. Characteristics of these spaces are thoroughly examined. We then prove the existence of solutions to the `fully hyperdissipative' case of (\ref{projected strato Salt hyperdiss}), where $\beta = m$ instead of $m-1$ in the previous item \ref{labelous}. Solutions are obtained through a variational framework developed in [\cite{goodair2024weak}].

    \item Section \ref{section existence of strong solutions} contains the statements and proofs of the main results. The key step is in demonstrating improved estimates to reduce the required hyperdissipation from the existence result of Section \ref{section high order estimates}. The It\^{o}-Stratonovich conversion is efficiently demonstrated due to the abstract result of [\cite{goodair2022stochastic}].

    \item Several appendices conclude this paper: Section \ref{section appendix I} presents the variational framework used to obtain solutions of the fully hyperdissipative equation, Section \ref{section maximal solution} states a second framework used to deduce higher order smoothness of solutions, Section \ref{Appendix II} gives the rigorous It\^{o}-Stratonovich conversion for an unbounded noise in infinite dimensions, and Section \ref{last appendix} contains the short proof that the Galerkin Projections explode unless restricted to a domain satisfying the boundary condition. 
\end{itemize}

\section{Preliminaries} \label{section preliminaries}

\subsection{Elementary Notation} \label{sub elementary notation}

In the following $\mathscr{O} \subset \R^2$ will be a smooth bounded domain endowed with Euclidean norm and Lebesgue measure $\lambda$. We consider Banach Spaces as measure spaces equipped with their corresponding Borel $\sigma$-algebra. Let $(\mathcal{X},\mu)$ denote a general topological measure space, $(\mathcal{Y},\norm{\cdot}_{\mathcal{Y}})$ and $(\mathcal{Z},\norm{\cdot}_{\mathcal{Z}})$ be separable Banach Spaces, and $(\mathcal{U},\inner{\cdot}{\cdot}_{\mathcal{U}})$, $(\mathcal{H},\inner{\cdot}{\cdot}_{\mathcal{H}})$ be general separable Hilbert spaces. We introduce the following spaces of functions. 
\begin{itemize}
    \item $L^p(\mathcal{X};\mathcal{Y})$ is the class of measurable $p-$integrable functions from $\mathcal{X}$ into $\mathcal{Y}$, $1 \leq p < \infty$, which is a Banach space with norm $$\norm{\phi}_{L^p(\mathcal{X};\mathcal{Y})}^p := \int_{\mathcal{X}}\norm{\phi(x)}^p_{\mathcal{Y}}\mu(dx).$$ In particular $L^2(\mathcal{X}; \mathcal{Y})$ is a Hilbert Space when $\mathcal{Y}$ itself is Hilbert, with the standard inner product $$\inner{\phi}{\psi}_{L^2(\mathcal{X}; \mathcal{Y})} = \int_{\mathcal{X}}\inner{\phi(x)}{\psi(x)}_\mathcal{Y} \mu(dx).$$ In the case $\mathcal{X} = \mathscr{O}$ and $\mathcal{Y} = \R^2$ note that $$\norm{\phi}_{L^2(\mathscr{O};\R^2)}^2 = \sum_{l=1}^2\norm{\phi^l}^2_{L^2(\mathscr{O};\R)}, \qquad \phi = \left(\phi^1, \dots, \phi^2\right), \quad \phi^l: \mathscr{O} \rightarrow \R.$$ We denote $L^p(\mathscr{O};\R^2)$ by simply $L^p$.
    
\item $L^{\infty}(\mathcal{X};\mathcal{Y})$ is the class of measurable functions from $\mathcal{X}$ into $\mathcal{Y}$ which are essentially bounded. $L^{\infty}(\mathcal{X};\mathcal{Y})$ is a Banach Space when equipped with the norm $$ \norm{\phi}_{L^{\infty}(\mathcal{X};\mathcal{Y})} := \inf\{C \geq 0: \norm{\phi(x)}_Y \leq C \textnormal{ for $\mu$-$a.e.$ }  x \in \mathcal{X}\}.$$ We denote $L^\infty(\mathscr{O};\R^2)$ by simply $L^\infty$.

      \item $C(\mathcal{X};\mathcal{Y})$ is the space of continuous functions from $\mathcal{X}$ into $\mathcal{Y}$.

      
    \item $C^m(\mathscr{O};\R)$ is the space of $m \in \N$ times continuously differentiable functions from $\mathscr{O}$ to $\R$, that is $\phi \in C^m(\mathscr{O};\R)$ if and only if for every $2$ dimensional multi index $\alpha = \alpha_1, \alpha_2$ with $\abs{\alpha}\leq m$, $D^\alpha \phi \in C(\mathscr{O};\R)$ where $D^\alpha$ is the corresponding classical derivative operator $\partial_{x_1}^{\alpha_1} \partial_{x_2}^{\alpha_2}$.
    
    \item $C^\infty(\mathscr{O};\R)$ is the intersection over all $m \in \N$ of the spaces $C^m(\mathscr{O};\R)$.
    
    \item $C^m_0(\mathscr{O};\R)$ for $m \in \N$ or $m = \infty$ is the subspace of $C^m(\mathscr{O};\R)$ of functions which have compact support.
    
    \item $C^m(\mathscr{O};\R^2), C^m_0(\mathscr{O};\R^2)$ for $m \in \N$ or $m = \infty$ is the space of functions from $\mathscr{O}$ to $\R^2$ whose component mappings each belong to $C^m(\mathscr{O};\R), C^m_0(\mathscr{O};\R)$. These spaces are simply denoted by $C^m, C^m_0$ respectively.
    
        \item $W^{m,p}(\mathscr{O}; \R)$ for $1 \leq p < \infty$ is the sub-class of $L^p(\mathscr{O}, \R)$ which has all weak derivatives up to order $m \in \N$ also of class $L^p(\mathscr{O}, \R)$. This is a Banach space with norm $$\norm{\phi}^p_{W^{m,p}(\mathscr{O}, \R)} := \sum_{\abs{\alpha} \leq m}\norm{D^\alpha \phi}_{L^p(\mathscr{O}; \R)}^p,$$ where $D^\alpha$ is the corresponding weak derivative operator. In the case $p=2$ the space $W^{m,2}(\mathscr{O}, \R)$ is Hilbert with inner product $$\inner{\phi}{\psi}_{W^{m,2}(\mathscr{O}; \R)} := \sum_{\abs{\alpha} \leq m} \inner{D^\alpha \phi}{D^\alpha \psi}_{L^2(\mathscr{O}; \R)}.$$
    
    \item $W^{m,\infty}(\mathscr{O};\R)$ for $m \in \N$ is the sub-class of $L^\infty(\mathscr{O}, \R)$ which has all weak derivatives up to order $m \in \N$ also of class $L^\infty(\mathscr{O}, \R)$. This is a Banach space with norm $$\norm{\phi}_{W^{m,\infty}(\mathscr{O}, \R)} := \sup_{\abs{\alpha} \leq m}\norm{D^{\alpha}\phi}_{L^{\infty}(\mathscr{O};\R^2)}.$$
    
   \item $W^{s,p}(\mathscr{O}; \R)$ for $0<s<1$ and $1 \leq p < \infty$ is the sub-class of functions $\phi \in L^p(\mathscr{O}, \R)$ such that $$\int_{\mathscr{O} \times \mathscr{O}} \frac{\abs{\phi(x)-\phi(y)}^p}{\abs{x-y}^{sp+2}}d\lambda(x,y) < \infty.$$ This is a Banach space with respect to the norm $$\norm{\phi}_{W^{s,p}(\mathscr{O}; \R)}^p= \norm{\phi}_{L^p(U;\R)}^p + \int_{\mathscr{O} \times \mathscr{O}} \frac{\abs{\phi(x)-\phi(y)}^p}{\abs{x-y}^{sp+2}}d\lambda(x,y).$$ For $p=2$ this is a Hilbert space with inner product $$\inner{\phi}{\psi}_{W^{s,2}(\mathscr{O}; \R)} = \inner{\phi}{\psi}_{L^2(\mathscr{O};\R)}+ \int_{\mathscr{O} \times \mathscr{O}} \frac{\big(\phi(x)-\phi(y)\big)\big(\psi(x)-\psi(y)\big)}{\abs{x-y}^{2s+2}}d\lambda(x,y).$$
    
    \item $W^{s,p}(\mathscr{O}; \R)$ for $1 \leq s < \infty,$ $s \not\in \N$ and $1 \leq p < \infty$ is, using the notation $\floor{s}$ to mean the integer part of $s$, the sub-class of $W^{\floor{s},p}(\mathscr{O}; \R)$ such that the distributional derivatives $D^\alpha\phi$ belong to $W^{s-\floor{s},p}(\mathscr{O}; \R)$ for every multi-index $\alpha$ such that $\abs{\alpha}=\floor{s}.$ This is a Banach space with norm $$\norm{\phi}_{W^{s,p}(\mathscr{O}; \R)}^p= \norm{\phi}_{W^{\floor{s},p}(\mathscr{O};\R)}^p + \sum_{\abs{\alpha}=\floor{s}} \int_{\mathscr{O} \times \mathscr{O}} \frac{\abs{D^\alpha\phi(x)-D^\alpha\phi(y)}^p}{\abs{x-y}^{(s-\floor{s})p+2}}dxdy$$ and a Hilbert space in the case $p=2$, with inner product $$\inner{\phi}{\psi}_{W^{s,2}(\mathscr{O}; \R)} = \inner{\phi}{\psi}_{W^{\floor{s},2}(\mathscr{O};\R)}+ \int_{\mathscr{O} \times \mathscr{O}} \frac{\big(D^\alpha\phi(x)-D^\alpha\phi(y)\big)\big(D^\alpha\psi(x)-D^\alpha\psi(y)\big)}{\abs{x-y}^{2(s-\floor{s})+2}}d\lambda(x,y).$$

 \item $W^{s,p}(\mathscr{O}; \R^2)$ for $s > 0$ and $1 \leq p < \infty$ is the Banach space of functions $\phi:\mathscr{O} \rightarrow \R^2$ whose components $(\phi^l)$ are each elements of the space $W^{s,p}(\mathscr{O}; \R).$ We denote this space by simply $W^{s,p}$. The associated norm is $$
    \norm{\phi}_{W^{s,p}}^p= \sum_{l=1}^2\norm{\phi^l}^p_{W^{s,p}(\mathscr{O}; \R)}
    $$ and similarly when $p=2$ this space is Hilbert with inner product $$\inner{\phi}{\psi}_{W^{s,2}} = \sum_{l=1}^2\inner{\phi^l}{\psi^l}_{W^{s,2}(\mathscr{O}, \R)}.$$
    
          \item $W^{m,\infty}(\mathscr{O}; \R^2)$ is the sub-class of $L^\infty(\mathscr{O}, \R^2)$ which has all weak derivatives up to order $m \in \N$ also of class $L^\infty(\mathscr{O}, \R^2)$. This is a Banach space with norm $$\norm{\phi}_{W^{m,\infty}} := \sup_{l \leq N}\norm{\phi^l}_{W^{m,\infty}(\mathscr{O}; \R)}.$$

    \item $W^{m,p}_0(\mathscr{O};\R), W^{m,p}_0(\mathscr{O};\R^2)$ for $m \in \N$ and $1 \leq p \leq \infty$ is the closure of $C^\infty_0(\mathscr{O};\R),C^\infty_0(\mathscr{O};\R^2)$ in $W^{m,p}(\mathscr{O};\R), W^{m,p}(\mathscr{O};\R^2)$.

    \item $\mathscr{L}(\mathcal{Y};\mathcal{Z})$ is the space of bounded linear operators from $\mathcal{Y}$ to $\mathcal{Z}$. This is a Banach Space when equipped with the norm $$\norm{F}_{\mathscr{L}(\mathcal{Y};\mathcal{Z})} = \sup_{\norm{y}_{\mathcal{Y}}=1}\norm{Fy}_{\mathcal{Z}}$$ and is simply the dual space $\mathcal{Y}^*$ when $\mathcal{Z}=\R$, with operator norm $\norm{\cdot}_{\mathcal{Y}^*}.$
    
     \item $\mathscr{L}^2(\mathcal{U};\mathcal{H})$ is the space of Hilbert-Schmidt operators from $\mathcal{U}$ to $\mathcal{H}$, defined as the elements $F \in \mathscr{L}(\mathcal{U};\mathcal{H})$ such that for some basis $(e_i)$ of $\mathcal{U}$, $$\sum_{i=1}^\infty \norm{Fe_i}_{\mathcal{H}}^2 < \infty.$$ This is a Hilbert space with inner product $$\inner{F}{G}_{\mathscr{L}^2(\mathcal{U};\mathcal{H})} = \sum_{i=1}^\infty \inner{Fe_i}{Ge_i}_{\mathcal{H}}$$ which is independent of the choice of basis.



\end{itemize}

We next give the probabilistic set up. Let $(\Omega,\mathcal{F},(\mathcal{F}_t), \mathbbm{P})$ be a fixed filtered probability space satisfying the usual conditions of completeness and right continuity. We take $\mathcal{W}$ to be a cylindrical Brownian motion over some Hilbert Space $\mathfrak{U}$ with orthonormal basis $(e_i)$. Recall (e.g. [\cite{lototsky2017stochastic}], Definition 3.2.36) that $\mathcal{W}$ admits the representation $\mathcal{W}_t = \sum_{i=1}^\infty e_iW^i_t$ as a limit in $L^2(\Omega;\mathfrak{U}')$ whereby the $(W^i)$ are a collection of i.i.d. standard real valued Brownian Motions and $\mathfrak{U}'$ is an enlargement of the Hilbert Space $\mathfrak{U}$ such that the embedding $J: \mathfrak{U} \rightarrow \mathfrak{U}'$ is Hilbert-Schmidt and $\mathcal{W}$ is a $JJ^*-$cylindrical Brownian Motion over $\mathfrak{U}'$. Given a process $F:[0,T] \times \Omega \rightarrow \mathscr{L}^2(\mathfrak{U};\mathscr{H})$ progressively measurable and such that $F \in L^2\left(\Omega \times [0,T];\mathscr{L}^2(\mathfrak{U};\mathscr{H})\right)$, for any $0 \leq t \leq T$ we define the stochastic integral $$\int_0^tF_sd\mathcal{W}_s:=\sum_{i=1}^\infty \int_0^tF_s(e_i)dW^i_s,$$ where the infinite sum is taken in $L^2(\Omega;\mathscr{H})$. We can extend this notion to processes $F$ which are such that $F(\omega) \in L^2\left( [0,T];\mathscr{L}^2(\mathfrak{U};\mathscr{H})\right)$ for $\mathbbm{P}-a.e.$ $\omega$ via the traditional localisation procedure. In this case the stochastic integral is a local martingale in $\mathscr{H}$. \footnote{A complete, direct construction of this integral, a treatment of its properties and the fundamentals of stochastic calculus in infinite dimensions can be found in [\cite{prevot2007concise}] Section 2.}

\subsection{Functional Framework} \label{functional framework subsection}

We now recap the classical functional framework for the study of the deterministic Navier-Stokes Equation. We formally define the operator $\mathcal{L}$ as well as the divergence-free and Navier boundary conditions. The nonlinear operator $\mathcal{L}$ is defined for sufficiently regular functions $f,g:\mathscr{O} \rightarrow \R^2$ by $\mathcal{L}_fg:= \sum_{j=1}^2f^j\partial_jg.$ Here and throughout the text we make no notational distinction between differential operators acting on a vector valued function or a scalar valued one; that is, we understand $\partial_jg$ by its component mappings $(\partial_lg)^l := \partial_jg^l$. For any $m \geq 1$, the mapping $\mathcal{L}: W^{m+1,2} \rightarrow W^{m,2}$ defined by $f \mapsto \mathcal{L}_ff$ is continuous (see [\cite{goodair2022navier}] Lemma 1.2). Some more technical properties of the operator are given at the end of this subsection. For the divergence-free condition we mean a function $f$ such that the property $$\textnormal{div}f := \sum_{j=1}^2 \partial_jf^j = 0$$ holds. We require this property and the boundary condition to hold for our solution $u$ at all times, understood in their traditional weak senses: that is for weak derivatives $\partial_j$ so $\sum_{j=1}^2 \partial_jf^j = 0$ holds as an identity in $L^2(\mathscr{O};\R)$ whilst the boundary condition is understood in terms of trace. To be precise we first define the restriction mapping on functions $f \in W^{1,2}(\mathscr{O};\R) \cap C(\bar{\mathscr{O}};\R)$ by the restriction of $f$ to the boundary $\partial \mathscr{O}$, which is then shown to be a bounded operator into $W^{\frac{1}{2},2}(\partial \mathscr{O};\R)$ (see Lemma \ref{bounded trace} and more classical sources of e.g. [\cite{evans2010partial}]). By the density of $C(\bar{\mathscr{O}};\R)$ then the trace operator is well defined on the whole of $W^{1,2}(\mathscr{O};\R)$ as a continuous linear extension of the restriction mapping, and furthermore on $W^{1,2}(\mathscr{O};\R^2)$ by the trace of the components. The rate of strain tensor $D$ appearing in (\ref{navier boundary conditions}) is a mapping $D: W^{1,2} \rightarrow L^{2}(\mathscr{O};\R^{2 \times 2})$ defined by $$f \mapsto \begin{bmatrix}
        \partial_1f^1 & \frac{1}{2}\left(\partial_1f^2 + \partial_2f^1\right)\\
        \frac{1}{2}\left(\partial_2f^1 + \partial_1f^2\right) & \partial_2f^2
    \end{bmatrix}
    $$
    or in component form, $(Df)^{k,l}:=\frac{1}{2}\left(\partial_kf^l + \partial_lf^k\right)$. Note that if $f \in W^{2,2}$ then $Df \in W^{1,2}(\mathscr{O};\R^{2 \times 2})$ so the trace of its components is well defined and henceforth we understand the boundary condition $$ 2(Df)\mathbf{n} \cdot \iota + \alpha f \cdot \iota = 0$$ on $\partial \mathscr{O}$ to be in this trace sense. The same is true for $f \cdot \mathbf{n} = 0$. We look to impose these conditions by incorporating them into the function spaces where our solution takes value, and will always assume that $\alpha \in C^2(\partial \mathscr{O};\R)$ so as to match the regularity required in [\cite{clopeau1998vanishing}].
\begin{definition}
We define $C^{\infty}_{0,\sigma}$ as the subset of $C^{\infty}_0$ of functions which are divergence-free. $L^2_\sigma$ is defined as the completion of $C^{\infty}_{0,\sigma}$ in $L^2$, whilst we introduce $\bar{W}^{1,2}_\sigma$ as the intersection of $W^{1,2}$ with $L^2_\sigma$ and $\bar{W}^{2,2}_{\alpha}$ by $$\bar{W}^{2,2}_{\alpha}:= \left\{f \in W^{2,2}\cap \bar{W}^{1,2}_{\sigma}: 2(Df)\mathbf{n} \cdot \iota + \alpha f \cdot \iota = 0 \textnormal{ on } \partial \mathscr{O}\right\}.$$
\end{definition}

\begin{remark} \label{new first labelled remark}
    $L^2_{\sigma}$ can be characterised as the subspace of $L^2$ of weakly divergence-free functions with normal component weakly zero at the boundary (see [\cite{robinson2016three}] Lemma 2.14). Moreover the complement space of $L^2_{\sigma}$ in $L^2$ is characterised as the subspace of $L^2$ of functions $f$ such that there exists a $g \in W^{1,2}(\mathscr{O};\R)$ with the property that $f = \nabla g$ (see [\cite{robinson2016three}] Theorem 2.16).
\end{remark}

\begin{remark}  \label{first labelled remark}
    $\bar{W}^{1,2}_{\sigma}$ is precisely the subspace of $W^{1,2}$ consisting of divergence-free functions $f$ such that $f \cdot \mathbf{n} = 0$ on $\partial \mathscr{O}$. Moreover as both $D: W^{2,2} \rightarrow W^{1,2}(\mathscr{O};\R^{2 \times 2})$ and the trace mapping $W^{1,2}(\mathscr{O};\R) \rightarrow L^2(\partial \mathscr{O} ; \R)$ are continuous, then $\bar{W}^{1,2}_{\sigma}$, $\bar{W}^{2,2}_{\alpha}$ are closed in the $W^{1,2}$, $W^{2,2}$ norms respectively.
\end{remark}

We note that the Poincar\'{e} Inequality holds for the component mappings of functions in $\bar{W}^{1,2}_{\sigma}$. The inequality (see e.g. [\cite{robinson2016three}] Theorem 1.9) holds for the component mapping $f^j$ of $f \in \bar{W}^{1,2}_{\sigma}$ if $$ \int_{\mathscr{O}}f^j d\lambda = 0,$$ and observe that via an approximation with the set $C^{\infty}_{0,\sigma}$ which is dense in $L^2_{\sigma}$, one can conclude that for all $g \in L^2_{\sigma}$ and $\psi \in W^{1,2}(\mathscr{O};\R)$, $$\inner{g}{\nabla \psi} = 0$$ (this is the statement of [\cite{robinson2016three}] Lemma 2.11). Moreover by choosing $\phi$ as the function $\phi(x^1,x^2) := x^j$, then $$\int_{\mathscr{O}}f^j d\lambda = \inner{f}{\nabla \phi} = 0$$ so the Poincar\'{e} Inequality holds, and as such we equip $\bar{W}^{1,2}_{\sigma}$ with the inner product $$\inner{f}{g}_1 := \sum_{j=1}^2 \inner{\partial_j f}{\partial_j g}$$ which is equivalent to the full $W^{1,2}$ one. We introduce the Leray Projector $\mathcal{P}$ as the orthogonal projection in $L^2$ onto $L^2_{\sigma}$. It is well known (see e.g. [\cite{temam2001navier}] Remark 1.6.) that for any $m \in \N$, $\mathcal{P}$ is continuous as a mapping $\mathcal{P}: W^{m,2} \rightarrow W^{m,2}$. Through $\mathcal{P}$ we define the Stokes Operator $A: W^{2,2} \rightarrow L^2_{\sigma}$ by $A:= -\mathcal{P}\Delta$. We understand the Laplacian as an operator on vector valued functions through the component mappings, $(\Delta f)^l := \Delta f^l$. From the continuity of $\mathcal{P}$ we have immediately that for $m \in \{0\} \cup \N$, $A: W^{m+2,2} \rightarrow W^{m,2}$ is continuous. We also note that $A = A\mathcal{P}$ as the Laplacian leaves the complement space of $L^2_{\sigma}$ invariant, see e.g. [\cite{goodair20233d}] Subsection 2.2 and Lemma 2.7, and by iterating this property then for $\beta \in \N$ we have that $\mathcal{P}(-\Delta)^\beta = A^{\beta}$.

\begin{lemma} \label{eigenfunctions for navier}
    There exists a collection of functions $(\bar{a}_k)$, $\bar{a}_k \in C^{\infty}(\bar{\mathscr{O}};\R^2) \cap \bar{W}^{2,2}_{\alpha}$, such that the $(\bar{a}_k)$ are eigenfunctions of $A$, are an orthonormal basis in $L^2_{\sigma}$ and a basis in $\bar{W}^{1,2}_{\sigma}$. Moreover the eigenvalues $(\bar{\lambda}_k)$ are strictly positive and approach infinity as $k \rightarrow \infty$.
\end{lemma}

\begin{proof}
    This is the content of [\cite{clopeau1998vanishing}] Lemma 2.2, where $(\bar{a}_k)$ is $(v_k)$ in their notation. The fact that this system consists of eigenfunctions of $A$ follows from (2.10) by taking $\mathcal{P}$ of the top line: that is, $$\bar{\lambda}_k\bar{a}_k = \mathcal{P}\left(\bar{\lambda}_k\bar{a}_k\right) = \mathcal{P}\left(-\Delta \bar{a}_k + \nabla \pi_k  \right) = A \bar{a}_k$$ using that $\mathcal{P}\nabla \pi_k = 0$ from Remark \ref{new first labelled remark}. Note that their result is stated for only $W^{3,2}$ regularity due to just a $C^2$ boundary, which is lifted to smooth eigenfunctions when the boundary is smooth (see for example, [\cite{temam2001navier}] Chapter 1 Section 2.6). 
\end{proof}

One can see that we are building a framework to parallel that of the classic no-slip case, though a significant difference comes in the presence of a boundary integral for Green's type identities. What we achieve now is the following, recalling $\kappa \in C^2(\partial \mathscr{O};\R)$ to be the curvature of $\partial \mathscr{O}$.

\begin{lemma} \label{greens for navier}
    For $f \in \bar{W}^{2,2}_{\alpha}$, $\phi \in \bar{W}^{1,2}_{\sigma}$, we have that $$\inner{\Delta f}{\phi}  = -\inner{f}{\phi}_1 + \inner{(\kappa - \alpha)f}{\phi}_{L^2(\partial \mathscr{O}; \R^2)}.$$
\end{lemma}

\begin{proof}
    This is demonstrated in [\cite{kelliher2006navier}] equation (5.1).
\end{proof}

Due to this boundary integral we will make great use of Lebesgue spaces and fractional Sobolev spaces defined on the boundary $\partial \mathscr{O}$. The spaces $L^p(\partial\mathscr{O};\R^2)$ can be defined precisely as in Subsection \ref{sub elementary notation} for $\partial \mathscr{O}$ equipped with its surface measure, and in fact the same is true for $W^{s,2}(\partial \mathscr{O};\R)$ with $0 < s < 1$ and hence $W^{s,2}(\partial \mathscr{O};\R^2)$ in the same manner. Indeed this definition is given in [\cite{grisvard2011elliptic}] pp.20, where it is shown to be equivalent to the often used definition locally as the space $W^{s,2}(\R;\R)$ via a coordinate transformation and partition of unity. The stability under a change of variables ensures results of H\"{o}lder's Inequality and (one dimensional) Sobolev Embeddings hold on the boundary for these spaces. The relation to the trace of functions in $\mathscr{O}$ is stated now.
\begin{lemma} \label{bounded trace}
    For $\frac{1}{2} <s < \frac{3}{2}$ the trace operator is bounded and linear from $W^{s,2}(\mathscr{O};\R)$ into $W^{s- \frac{1}{2},2}(\partial \mathscr{O};\R)$. 
\end{lemma}

\begin{proof}
    See [\cite{ding1996proof}] Theorem 1. 
\end{proof}

We shall make use of this result for the characterisation of the fractional Sobolev spaces as interpolation spaces. In fact the renowned book of Adams [\cite{adams2003sobolev}] defines these spaces in this way (7.57, pp.250), and their equivalence is well understood (see for example [\cite{tartar2007introduction}] pp.83). A proof of this equivalence relies on a norm preserving extension operator for the space as defined by interpolation (which follows from the classical integer valued case, see [\cite{adams2003sobolev}] Theorem 5.24 and [\cite{stein1970singular}] chapter 6), the equivalence of the interpolation space on $\R^2$ with that defined by Fourier transformations (see [\cite{adams2003sobolev}] 7.63 pp.252), the further equivalence of this space with our definition on $\R^2$ ([\cite{demengel2012functional}] Proposition 4.17), and finally a norm preserving extension for this fractional Sobolev space ([\cite{di2012hitchhikers}] Theorem 5.4). Therefore there exists a constant $c$ such that for $f \in W^{1,2}$ and $0 < s < 1$, we have that \begin{equation} \label{interpolation}
    \norm{f}_{W^{s,2}} \leq c\norm{f}^{1-s}\norm{f}_{W^{1,2}}^s.
\end{equation}
In particular for $s=\frac{1}{2}$, if the result of Lemma \ref{bounded trace} were true in this limiting case (that is, an embedding of $W^{\frac{1}{2},2}(\mathscr{O};\R^2)$ into $L^2(\partial \mathscr{O};\R^2)$) then combined with (\ref{interpolation}) we would obtain \begin{equation}
    \label{inequality from Lions}
    \norm{f}_{L^2(\partial \mathscr{O};\R^2)}^2 \leq c\norm{f}\norm{f}_{W^{1,2}}.
\end{equation}
This inequality is in fact true and is classical in the study of our problem (see for example [\cite{lions1996mathematical}] pp.130, [\cite{kelliher2006navier}] equation (2.5)) though some additional machinery is required to prove it. In short the result can be achieved by showing that the trace operator is a continuous linear operator from an appropriate Besov space which similarly interpolates between $L^2$ and $W^{1,2}$ (see [\cite{adams2003sobolev}] Theorem 7.43 and Remark 7.45 for the trace embedding, and the Besov Spaces subchapter for the interpolation). Moving on we introduce the finite dimensional projections $(\bar{\mathcal{P}}_n)$, where $\bar{\mathcal{P}}_n$ is the orthogonal projection onto $\bar{V}_n:=\textnormal{span}\{\bar{a}_1, \dots, \bar{a}_n\}$ in $L^2$ (note that $\bar{V}_n$ is a Hilbert Space equipped with any $W^{k,2}$). That is, $\bar{\mathcal{P}}_n$ is given by $$\bar{\mathcal{P}}_n : f \mapsto \sum_{k=1}^n \inner{f}{\bar{a}_k}\bar{a}_k.$$

\newpage



    



\newpage

It will also be of use to us to consider the vorticity in this context, which we do by introducing the operator $\textnormal{curl}: W^{1,2} \rightarrow L^2(\mathscr{O};\R)$ by $$\textnormal{curl}: f \rightarrow \partial_1 f^2 - \partial_2 f^1.$$
A significant property in the study of vorticity is the following.
\begin{lemma} \label{lemma for curl and P}
    For all $f \in W^{1,2}$, $\textnormal{curl}(\mathcal{P}f) = \textnormal{curl}f$.
\end{lemma}
\begin{proof}
     If $\phi = \nabla g \in W^{1,2}$ then $\textnormal{curl}\phi = \partial_1\partial_2g - \partial_2\partial_1g = 0$, which from Remark \ref{new first labelled remark} establishes that $\textnormal{curl}\left(\mathcal{P}f\right) = \textnormal{curl}\left([\mathcal{P} + \mathcal{P}^{\perp}]f \right) = \textnormal{curl}f$ where $\mathcal{P}^{\perp}$ is the complement projection $I - \mathcal{P}$ on $L^2$.
\end{proof}
The curl is intrinsically related to the Navier boundary conditions through $\kappa$, by the relation proven as Lemma 2.1 in [\cite{clopeau1998vanishing}] which we state here.
\begin{lemma}
    For all $f \in W^{2,2}$ with $f \cdot \mathbf{n} = 0$ on $\partial \mathscr{O}$, we have that $$2(Df)\mathbf{n} \cdot \iota +2\kappa f \cdot \iota - \textnormal{curl}f = 0$$ on $\partial \mathscr{O}$. 
\end{lemma}


Moreover we can make the condition (\ref{its a new rep}) discussed in the introduction precise.

\begin{corollary} \label{vorticity corollary}
    Suppose that $f \in W^{2,2}$ with $f \cdot \mathbf{n} = 0$ on $\partial \mathscr{O}$. Then $f$ satisfies $$2(Df)\mathbf{n} \cdot \mathbf{\iota} + \alpha f\cdot \mathbf{\iota} = 0$$ on $\partial \mathscr{O}$ if and only if it satisfies $$ \textnormal{curl}f = (2\kappa - \alpha)f \cdot \iota$$ on $\partial \mathscr{O}$.
\end{corollary}

 Returning to the nonlinear term, we shall frequently understand the $L^2$ inner product as a duality pairing between $L^{\frac{4}{3}}$ and $L^4$ as justified in the following.

\begin{lemma} \label{the 2D bound lemma}
There exists a constant $C$ such that for every $\phi,f,g \in W^{1,2}$, we have that $$ \left\vert  \inner{\mathcal{L}_{\phi}f}{g} \right\vert \leq C\norm{\phi}^{\frac{1}{2}}\norm{\phi}_1^{\frac{1}{2}}\norm{f}_1\norm{g}^{\frac{1}{2}}\norm{g}_1^{\frac{1}{2}}.$$
\end{lemma}

\begin{proof}
   From two instances of H\"{o}lder's Inequality as well as the Gagliardo-Nirenberg Inequality with $p=4, q=2, \alpha = 1/2$ and $m=1$, 
  \begin{align} \nonumber
       \left\vert\inner{\mathcal{L}_{\phi}f}{g}\right\vert \leq \norm{\mathcal{L}_{\phi}f}_{L^{4/3}}\norm{g}_{L^4} &\leq c\sum_{k=1}^2\norm{\phi}_{L^4}\norm{\partial_kf}\norm{g}_{L^4} \leq c\norm{\phi}^{\frac{1}{2}}\norm{\phi}_1^{\frac{1}{2}}\norm{f}_1\norm{g}^{\frac{1}{2}}\norm{g}_1^{\frac{1}{2}}\label{a bound in align}
    \end{align}
\end{proof}

\begin{remark}
    This inequality is critical in the study of 2D Navier-Stokes, its failure in 3D responsible for the lack of global strong solutions.
\end{remark}

This subsection concludes with a symmetry result for the trilinear form defined by the nonlinear term which is classical when zero trace is assumed, but perhaps not in lieu of this assumption.

\begin{lemma} \label{navier boundary nonlinear}
    For every $\phi \in \bar{W}^{1,2}_{\sigma}$, $f,g \in W^{1,2}$ we have that \begin{equation}\label{wloglhs}\inner{\mathcal{L}_{\phi}f}{g}= -\inner{f}{\mathcal{L}_{\phi}g}.\end{equation}
    Moreover, \begin{equation} \label{cancellationproperty'} \inner{\mathcal{L}_{\phi}f}{f}= 0.\end{equation}
\end{lemma}

\begin{proof}
    Of course (\ref{cancellationproperty'}) follows from (\ref{wloglhs}) with the choice $g:= f$ and symmetry of the inner product, so we just show (\ref{wloglhs}). As stated it is classical that the result holds in the case that $\phi$ has zero trace (see e.g. [\cite{robinson2016three}] Lemma 3.2) by an approximation in $W^{1,2}$ of compactly supported functions, though without that we have to do the integration by parts directly:
    \begin{align*}
    \inner{\mathcal{L}_{\phi} f}{g} &= \sum_{j=1}^2\sum_{l=1}^2\inner{\phi^j\partial_jf^l}{g^l}_{L^2(\mathscr{O};\R)}\\
    &= \sum_{j=1}^2\sum_{l=1}^2\left(\inner{\phi^j\partial_jf^l}{g^l}_{L^2(\mathscr{O};\R)} + \inner{\partial_j\phi^jf^l}{g^l}_{L^2(\mathscr{O};\R)}\right)\\
    &= \sum_{j=1}^2\sum_{l=1}^2\inner{\partial_j(\phi^jf^l)}{g^l}_{L^2(\mathscr{O};\R)}\\
    &= -\sum_{j=1}^2\sum_{l=1}^2\inner{\phi^jf^l}{\partial_jg^l}_{L^2(\mathscr{O};\R)} + \sum_{j=1}^2\sum_{l=1}^2\inner{\phi^jf^l}{g^l\mathbf{n}^j}_{L^2(\partial \mathscr{O};\R)}\\
    &= -\inner{f}{\mathcal{L}_{\phi} g}
\end{align*}
    where we have used that $\sum_{j=1}^2\partial_j \phi^j = 0$ (divergence-free) in $\mathscr{O}$ and $\sum_{j=1}^2\phi^j \mathbf{n}^j = 0$ ($\phi \cdot \mathbf{n}=0$) on $\partial \mathscr{O}$.

\end{proof}

\subsection{Stochastic Navier-Stokes Equations} \label{sub salt}

We first introduce the diffusion operator $B$ from the equations (\ref{number2equationSALT}), (\ref{projected strato Salt}), (\ref{number2equationSALT hyper}) and (\ref{projected strato Salt hyperdiss}) by its action on the basis vectors $(e_i)$ of $\mathfrak{U}$. This is sufficient to define $B$ on the entirety of $\mathfrak{U}$ as shown in [\cite{goodair2022stochastic}] Subsection 2.2. Denoting $B(e_i) = B_i$, define
$$B_i:f \mapsto \mathcal{L}_{\xi_i}f + \mathcal{T}_{\xi_i}f, \qquad  \mathcal{T}_{g}f := \sum_{j=1}^2 f^j\nabla g^j$$ where we shall always impose that $\xi_i \in L^2_{\sigma}$, along with some smoothness and decay stated in the relevant results hereafter. Whilst we do prove existence results for the genuine Stratonovich forms, it is much preferred to work with the corresponding It\^{o} forms
\begin{equation} \label{projected Ito Salt}
    u_t = u_0 - \int_0^t\mathcal{P}\mathcal{L}_{u_s}u_s\ ds - \nu\int_0^t A u_s\, ds + \frac{1}{2}\int_0^t\sum_{i=1}^\infty \mathcal{P}B_i^2u_s ds - \int_0^t \mathcal{P}Bu_s d\mathcal{W}_s 
\end{equation}
and
\begin{equation} \label{projected Ito Salt hyperdiss}
  u_t = u_0 - \int_0^t\mathcal{P}\mathcal{L}_{u_s}u_s\ ds - \nu\int_0^t A^\beta u_s\, ds + \frac{1}{2}\int_0^t\sum_{i=1}^\infty \mathcal{P}B_i^2u_s ds - \int_0^t \mathcal{P}Bu_s d\mathcal{W}_s
\end{equation}
for a conversion motivated by Theorem \ref{theorem for ito strat conversion} in the appendix. Fundamental properties of the operator $B_i$ are proven in [\cite{goodair20233d}] Subsection 2.3, some of which we list now, though a complete description is deferred to [\cite{goodair20233d}]. Firstly for $k = 0, 1, 2, \dots$, there exists a constant $c$ such that  
\begin{align}
    \label{T_ibound}\norm{\mathcal{T}_{\xi_i}f}_{W^{k,2}}^2 &\leq  c \norm{\xi_i}^2_{W^{k+1,\infty}}\norm{f}^2_{W^{k,2}}\\
    \label{L_ibound} \norm{\mathcal{L}_{\xi_i}f}_{W^{k,2}}^2 &\leq c\norm{\xi_i}^2_{W^{k,\infty}}\norm{f}^2_{W^{k+1,2}}\\
    \label{boundsonB_i} \norm{B_if}_{W^{k,2}}^2 &\leq c\norm{\xi_i}^2_{W^{k+1,\infty}}\norm{f}^2_{W^{k+1,2}}
\end{align}
for $f, \xi$ as required by the right hand side. Moreover for $\xi_i \in W^{1,\infty}$, $\mathcal{T}_{\xi_i}$ is a bounded linear operator on $L^2$ so has adjoint $\mathcal{T}_{\xi_i}^*$ satisfying the same boundedness. In conjunction with property (\ref{wloglhs}), $\mathcal{L}_{\xi_i}$ is a densely defined operator in $L^2$ with domain of definition $W^{1,2}$, and has adjoint $\mathcal{L}_{\xi_i}^*$ in this space given by $-\mathcal{L}_{\xi_i}$ with same dense domain of definition. Likewise then $B_i^*$ is the densely defined adjoint $-\mathcal{L}_{\xi_i} + \mathcal{T}_{\xi_i}^*$. We also note from [\cite{goodair20233d}] Lemma 2.7 that $\mathcal{P}B_i = \mathcal{P}B_i\mathcal{P}$ hence $\mathcal{P}B_i^2 = (\mathcal{P}B_i)^2$, which is critical in the equivalence of the forms (\ref{projected Ito Salt}), (\ref{number2equationSALT}) as well as the high order estimates of Subsection \ref{subs noise estimates}. We also recall [\cite{goodair20233d}] Proposition 5.2, where the commutator $$[\Delta,B_i]:= \Delta B_i - B_i\Delta$$ was explicitly shown to be of second order, bounded by a constant $c$ such that for every $f\in W^{3,2}$, \begin{equation} \label{commutator bound} \norm{[\Delta,B_i]f}^2 \leq c\norm{\xi_i}_{W^{3,\infty}}^2\norm{f}_{W^{2,2}}^2.\end{equation}
One can extend this result to higher orders, such that for $f \in W^{k+3,2}$, \begin{equation} \label{commutator bound higher order} \norm{[\Delta,B_i]f}_{W^{k,2}}^2 \leq c\norm{\xi_i}_{W^{k+3,\infty}}^2\norm{f}_{W^{k+2,2}}^2.\end{equation}
It is also shown in [\cite{goodair2023navier}] Subsection 2.3 that for $\xi_i \in W^{2,\infty}$ and $f \in W^{2,2}$, \begin{equation} \label{curl of Bi}
    \textnormal{curl}(B_if) = \mathcal{L}_{\xi_i}(\textnormal{curl}f).
\end{equation}
Moreover in [\cite{goodair20233d}] Proposition 2.6 the following conservation inequalities are proven:
    \begin{align}
    \inner{B_i^2f}{f}_{W^{k,2}} +  \norm{B_if}_{W^{k,2}}^2 &\leq c\norm{\xi_i}_{W^{k+2,\infty}}^2\norm{f}_{W^{k,2}}^2 \label{combinedterminenergyinequality},\\
    \inner{B_if}{f}_{W^{k,2}}^2 &\leq c\norm{\xi_i}^2_{W^{k+1,\infty}}\norm{f}^4_{W^{k,2}}. \label{finalboundinderivativeproof}
\end{align}

\section{High Order Estimates and a Variational Framework} \label{section high order estimates}

In Subsection \ref{subs frac} we introduce Hilbert Spaces defined through the spectrum of the Stokes Operator. The fact that the noise belongs to these spaces and satisfies sufficient estimates in them is demonstrated in Subsection \ref{subs noise estimates}. An application of the variational framework from [\cite{goodair2024weak}] for the fully hyperdissipative equation in these spaces is presented in Subsection \ref{subs reg of hyper}, obtaining strong solutions.

\subsection{Fractional Powers of the Stokes Operator} \label{subs frac}

We now introduce powers of the Stokes Operator $A$, alongside associated spaces on which inner products relative to $A$ can be defined. These will prove necessary to handle the Stokes Operator in energy estimates. 

\begin{definition}
    For every $s \geq 0$, we define $D(A^s)$ as the subspace of functions $f \in L^2_{\sigma}$ such that $$\sum_{k=1}^\infty \bar{\lambda}_k^{2s}\inner{f}{\bar{a}_k}^2 < \infty.$$
    We define the mapping $A^s: D(A^s) \rightarrow L^2_{\sigma}$ by $$A^s: f \mapsto \sum_{k=1}^\infty \bar{\lambda}_k^s\inner{f}{\bar{a}_k}\bar{a}_k$$ and associated inner product and norm on $D(A^s)$ by $$ \inner{f}{g}_{A^s} = \inner{A^sf}{A^sg}, \qquad \norm{f}_{A^s}^2 = \inner{f}{f}_{A^s}.$$
\end{definition}

\begin{remark}
    There is an implicit dependency on $\alpha$ in $A^s$ through $(\bar{a}_k)$, $(\bar{\lambda}_k)$. 
\end{remark}

Some trivial but important results are collected in Lemma \ref{lemma for passing stokes powers}.

\begin{lemma} \label{lemma for passing stokes powers}
We have the following:
\begin{enumerate}
    \item \label{first item} If $f \in D(A^s)$ then $f \in D(A^r)$ for every $0 \leq r < s$;
    \item \label{second item} Let $f,g \in D(A^{2s})$. Then for any $p,q \geq 0$ such that $p + q = 2s$, we have that $$\inner{f}{g}_{A^s} = \inner{A^pf}{A^qg};$$
    \item \label{third item} For $0 \leq r < s$ and $f \in D(A^r)$, we have that $f \in D(A^s)$ if and only if $A^rf \in D(A^{s-r})$;
    \item For $0 \leq r < s$ and $f \in D(A^s)$, we have that $A^sf = A^{s-r}A^rf = A^rA^{s-r}f$.
\end{enumerate}
    
\end{lemma}

\begin{proof}
    We prove the statements in turn:
    \begin{enumerate}
        \item As $(\bar{\lambda}_k) \rightarrow \infty$ then eventually the sequence is greater than $1$, so the tail end of the sum increases with $s$. 
        \item We have that \begin{align*}
        \inner{f}{g}_{A^s} &= \inner{\sum_{k=1}^\infty \bar{\lambda}_k^s\inner{f}{\bar{a}_k}\bar{a}_k}{\sum_{j=1}^\infty \bar{\lambda}_j^s\inner{g}{\bar{a}_j}\bar{a_j}} = \sum_{k=1}^\infty \bar{\lambda}_k^{2s}\inner{f}{\bar{a}_k}\inner{g}{\bar{a}_k}\\ &= \inner{\sum_{k=1}^\infty \bar{\lambda}_k^p\inner{f}{\bar{a}_k}\bar{a}_k}{\sum_{j=1}^\infty \bar{\lambda}_j^q\inner{g}{\bar{a}_j}\bar{a_j}} = \inner{A^pf}{A^qg}.
    \end{align*}
    \item Observe that
    \begin{align*}
        \sum_{k=1}^\infty \bar{\lambda}_k^{2s}\inner{f}{\bar{a}_k}^2 = \sum_{k=1}^\infty \bar{\lambda}_k^{2(s-r)}\bar{\lambda}_k^{2r}\inner{f}{\bar{a}_k}^2  = \sum_{k=1}^\infty \bar{\lambda}_k^{2(s-r)}\inner{f}{A^r\bar{a}_k}^2=\sum_{k=1}^\infty \bar{\lambda}_k^{2(s-r)}\inner{A^rf}{\bar{a}_k}^2
    \end{align*}
employing item \ref{second item}.
    \item Using that both $A^rf \in D(A^{s-r})$ and $A^{s-r}f \in D(A^r)$ from item \ref{first item}, the result holds as in the proof of \ref{third item}.
    \end{enumerate}

\end{proof}

We shall frequently use these properties without explicit reference to the lemma. For further analysis in these spaces, we establish some base cases.

\begin{lemma} \label{lemma for equivalence of 2 inner product}
    The bilinear form $\inner{f}{g}_2:= \inner{Af}{Ag}$ defines an inner product on $\bar{W}^{2,2}_{\alpha}$ equivalent to the standard $W^{2,2}$ inner product. 
\end{lemma}

\begin{remark}
    It is tempting to immediately say that $\inner{\cdot}{\cdot}_2 = \inner{\cdot}{\cdot}_{A^1}$, but we want to rigorously build the equivalence of $A^1$ and $A$. 
\end{remark}

\begin{proof}
    We must show the existence of constants $c_1$, $c_2$ such that for all $f \in \bar{W}^{2,2}_{\alpha}$, $$c_1\norm{f}_{W^{2,2}}^2 \leq \norm{f}_2^2 \leq c_2\norm{f}_{W^{2,2}}^2.$$
The constant $c_2$ can in fact be taken as $2$, as $$\norm{f}_2^2 \leq \norm{\Delta f}^2 \leq 2\norm{f}_{W^{2,2}}^2.$$ The existence of such a $c_1$ is much more challenging and relies on estimates of the Stokes equation with the Navier boundary conditions, which has been proven in [\cite{tapia2021stokes}] Theorem 5.10. We use, in their notation, that $\mathbf{f} = A\mathbf{u}$ is a solution of the Stokes problem with $\pi = 0$, which gives the result.
\end{proof}

Following the remark, we connect this space with $D(A^1)$.

\begin{proposition} \label{prop norm equivalence for W22}
    We have that $D(A^1) = \bar{W}^{2,2}_{\alpha}$, that $A^1 = A$ on this space, and that $\inner{\cdot}{\cdot}_{A^1}$ is equivalent to the standard $W^{2,2}$ inner product. 
\end{proposition}

\begin{proof}
    We begin by showing the inclusion both ways:
    \begin{itemize}
        \item[$\subseteq$:] Take $f \in D(A^1)$. To show that $f \in \bar{W}^{2,2}_{\alpha}$, it is sufficient to show that the sequence $(\bar{\mathcal{P}}_nf)$ is convergent in $\bar{W}^{2,2}_{\alpha}$, as the limit must agree with the limit in $L^2_{\sigma}$ which is $f$. In particular, it is sufficient to show that the sequence $(\bar{\mathcal{P}}_nf)$ is Cauchy in $\bar{W}^{2,2}_{\alpha}$, and from Lemma \ref{lemma for equivalence of 2 inner product} it is sufficient to show the Cauchy property in the $\norm{\cdot}_2$ norm. For any $m < n$,
        $$\norm{\bar{\mathcal{P}}_nf - \bar{\mathcal{P}}_mf}_2^2 = \sum_{k=m+1}^n\bar{\lambda}_k^2\inner{f}{\bar{a}_k}^2 \leq  \sum_{k=m+1}^\infty \bar{\lambda}_k^2\inner{f}{\bar{a}_k}^2$$
        which approaches zero as $m \rightarrow \infty$ given that $f \in D(A^1)$, justifying the Cauchy property.

        \item[$\supseteq$:] Take $f \in \bar{W}^{2,2}_{\alpha}$. Then $$\bar{\lambda}_k^2\inner{f}{\bar{a}_k}^2 =  \inner{f}{A\bar{a}_k}^2.$$
        From Lemma \ref{greens for navier}, we have both that
        $$\inner{f}{A\bar{a}_k} = -\inner{f}{\Delta \bar{a}_k} = \inner{f}{\bar{a}_k}_1 + \inner{f}{(\alpha - \kappa)\bar{a}_k}_{L^2(\partial \mathscr{O}; \R^2)}$$
        and 
        $$\inner{Af}{\bar{a}_k} = -\inner{\Delta f}{ \bar{a}_k} = \inner{f}{\bar{a}_k}_1 + \inner{(\alpha - \kappa)f}{\bar{a}_k}_{L^2(\partial \mathscr{O}; \R^2)}$$
        so in particular $$\inner{f}{A\bar{a}_k} =  \inner{Af}{\bar{a}_k}.$$
        In total then,
        $$\sum_{k=1}^\infty \bar{\lambda}_k^2\inner{f}{\bar{a}_k}^2 = \sum_{k=1}^\infty \inner{f}{A\bar{a}_k}^2 = \sum_{k=1}^\infty \inner{Af}{\bar{a}_k}^2.$$
        As $f \in \bar{W}^{2,2}_{\alpha}$ then $Af \in L^2_{\sigma}$, so the above sum is finite and the inclusion is proven.
    \end{itemize}
    We have shown that $D(A^1) = \bar{W}^{2,2}_{\alpha}$. The fact that $A^1 = A$ on this space follows from the argument in the first inclusion which shows that $(\bar{\mathcal{P}}_nf)$ is convergent to $f$ in $\bar{W}^{2,2}_{\alpha}$, so in particular $f$ has the representation $$f = \sum_{k=1}^\infty \inner{f}{\bar{a}_k}\bar{a}_k$$ as a limit in $\bar{W}^{2,2}_{\alpha}$. Thus $$Af = \sum_{k=1}^\infty \bar{\lambda}_k \inner{f}{\bar{a}_k}\bar{a}_k$$ as a limit in $L^2_{\sigma}$, which is equal to $A^1f$. The inner product equivalence now follows from Lemma \ref{lemma for equivalence of 2 inner product}.
\end{proof}
Accordingly, we shall denote $A^1$ by simply $A$.

\begin{proposition} \label{prop for W12 norm equivalence Ahalf}
    Let $\alpha \geq \kappa$ everywhere on $\partial \mathscr{O}$. Then $D(A^{\frac{1}{2}}) = \bar{W}^{1,2}_{\sigma}$ and $\inner{\cdot}{\cdot}_{A^{\frac{1}{2}}}$ is equivalent to $\inner{\cdot}{\cdot}_1$. 
\end{proposition}

\begin{proof}
    We again show the inclusion both ways:
    \begin{itemize}
        \item[$\subseteq$:] Through the same arguments as Proposition \ref{prop norm equivalence for W22}, the result would hold if we show the equivalence of $\inner{\cdot}{\cdot}_{A^{\frac{1}{2}}}$ and $\inner{\cdot}{\cdot}_1$ on $\bar{V}:=\bigcup_n \bar{V}_n$. To this end we take $f \in \bar{V}$. Then by Lemma \ref{lemma for passing stokes powers} followed by Lemma \ref{greens for navier},  $$\norm{f}_{A^{\frac{1}{2}}}^2 = \inner{f}{Af} = -\inner{f}{\Delta f} =  \norm{f}^2_1 + \inner{(\alpha - \kappa)f}{f}_{L^2(\partial \mathscr{O}; \R^2)}.$$
    As $\alpha \geq \kappa$ then certainly $\norm{f}_{A^{\frac{1}{2}}}^2 \geq \norm{f}_1^2$. For the reverse, $$\norm{f}_1^2 - \inner{(\kappa - \alpha)f}{f}_{L^2(\partial \mathscr{O};\R^2)} \leq \norm{f}_1^2 + c\norm{\kappa - \alpha}_{L^\infty(\partial \mathscr{O}; \R)}\norm{f}\norm{f}_1 \leq c\norm{f}_1^2 $$
    having used (\ref{inequality from Lions}), where our final constant depends on $\kappa$ and $\alpha$. This justifies the norm equivalence and hence the inclusion.

    \item[$\supseteq$:] Take $f \in \bar{W}^{1,2}_{\sigma}$. Then there exists a sequence $(\phi^n)$ in $\bar{V}$ convergent to $f$ in $\bar{W}^{1,2}_{\sigma}$, hence this sequence is Cauchy in $\bar{W}^{1,2}_{\sigma}$ and established in the previous inclusion, is Cauchy in $D(A^{\frac{1}{2}}).$ In particular, there exists some $m\in \N$ such that for all $n \geq m$, $$\sum_{k=1}^\infty \bar{\lambda}_{k}\inner{\phi^n - \phi^m}{\bar{a}_k}^2 \leq 1.$$
    We use the fact that $(\phi^n)$ is convergent to $f$ in $L^2_{\sigma}$ to deduce that $$\lim_{n \rightarrow \infty}\inner{\phi^n}{\bar{a}_k}^2 = \inner{f}{\bar{a}_k}^2.$$ So for every $l \in \N$,
    \begin{align*}
        \sum_{k=1}^l \bar{\lambda}_{k}\inner{f}{\bar{a}_k}^2 = \lim_{n \rightarrow \infty} \sum_{k=1}^l \bar{\lambda}_{k}\inner{\phi^n}{\bar{a}_k}^2 &\leq \lim_{n \rightarrow \infty}2 \sum_{k=1}^l \bar{\lambda}_{k}\left(\inner{\phi^n - \phi^m}{\bar{a}_k}^2 + \inner{\phi^m}{\bar{a}_k}^2 \right)\\ &\leq 2 + 2\norm{\phi^m}_{A^{\frac{1}{2}}}^2.
    \end{align*}
    Taking the limit in $l$ now gives the result. 
    \end{itemize}
    The equivalence of the norms on $D(A^{\frac{1}{2}}) = \bar{W}^{1,2}_{\sigma}$ now follows from their equivalence on $\bar{V}$ which is dense in the Banach space. 
\end{proof}

In the remaining results of this subsection, we assume that $\alpha \geq \kappa$ so that Proposition \ref{prop for W12 norm equivalence Ahalf} holds. 

\begin{proposition} \label{prop all norm equivalences DA}
    For every $s \geq 0$, $D(A^s)$ is a Hilbert Space equipped with the $\inner{\cdot}{\cdot}_{A^s}$ inner product. For every $m \in \N$, $\inner{\cdot}{\cdot}_{A^{\frac{m}{2}}}$ is equivalent to the standard $W^{m,2}$ inner product on $D(A^{\frac{m}{2}})$.
\end{proposition}

\begin{proof}
    The fact that $D(A^s)$ is complete for the $\norm{\cdot}_{A^s}$ norm is exactly as was demonstrated in the second part of Proposition \ref{prop for W12 norm equivalence Ahalf}, using that any Cauchy sequence in $D(A^s)$ is Cauchy hence convergent in $L^2_{\sigma}$. For the equivalence of norms, we have already demonstrated the base cases of $D(A^{\frac{1}{2}})$ and $D(A)$. We will build on this inductively, and first show that for every $m \in \N$ there exists a constant $c$ such that for every $f \in D(A^{\frac{m}{2}})$, $$\norm{f}_{A^{\frac{m}{2}}} \leq c\norm{\cdot}_{W^{m,2}}.$$
    Let us first consider the case where $m=2k$ for $k \in \N$. We make precise that $A^k$ is truly $(-\mathcal{P}\Delta)^k$ on $D(A^k)$, having established that $A^1 = A$ on $D(A^1)$ in Proposition \ref{prop norm equivalence for W22}, and inductively if $A^{k-1} = (-\mathcal{P}\Delta)^{k-1}$ on $D(A^{k-1})$, then for $f \in D(A^k)$ from Lemma \ref{lemma for passing stokes powers}, $A^kf = A^{k-1}Af = (-\mathcal{P}\Delta)^{k}f.$ Using that the Leray Projector is bounded on all $W^{l,2}(\mathscr{O};\R^2)$, and that $\Delta$ is bounded from $W^{l+2,2}(\mathscr{O};\R^2)$ into $W^{l,2}(\mathscr{O};\R^2)$, then the result is clear. In the case $m = 2k-1$ the argument is similar, with a brief additional step: $$\norm{f}_{A^{\frac{2k-1}{2}}} = \norm{A^{\frac{1}{2}}A^{k-1}f} \leq c\norm{A^{k-1}f}_1 \leq c\norm{f}_{W^{2k-1,2}}.$$
    To demonstrate the reverse inequality we again use estimates for the Stokes Equation from [\cite{tapia2021stokes}] Theorem 5.10, extended to higher orders for a smooth domain (see for example [\cite{constantin1988navier}] Remark 3.8), which is that for $f \in \bar{W}^{2,2}_{\alpha}$, \begin{equation} \label{which is htat for}\norm{f}_{W^{k+2,2}} \leq c\norm{Af}_{W^{k,2}}.\end{equation}
    We make the inductive assumption that for all $k \leq m$ and $g \in D(A^{\frac{k}{2}})$, $\norm{g}_{W^{k,2}} \leq c\norm{g}_{A^{\frac{k}{2}}}.$ Then for $f \in D(A^{\frac{m+1}{2}})$, firstly by (\ref{which is htat for}) and secondly through the inductive hypothesis with $g = Af \in D(A^{\frac{m-1}{2}})$, $$\norm{f}_{W^{m+1,2}} \leq c\norm{Af}_{W^{m-1,2}} \leq c\norm{Af}_{A^{\frac{m-1}{2}}} = c\norm{f}_{A^{\frac{m+1}{2}}}.$$

\end{proof}

\begin{lemma} \label{lemma for inclusion of high A}
    We have that $D(A^{\frac{3}{2}}) = W^{3,2} \cap \bar{W}^{2,2}_{\alpha}$, and for $m \in \N$, $m \geq 2$, that $D(A^{\frac{m}{2}}) \subseteq W^{m,2} \cap \bar{W}^{2,2}_{\alpha}$.
\end{lemma}

\begin{proof}
    The inclusion $D(A^{\frac{m}{2}}) \subseteq W^{m,2} \cap \bar{W}^{2,2}_{\alpha}$ is similar to Propositions \ref{prop norm equivalence for W22} and \ref{prop for W12 norm equivalence Ahalf}, where $(\bar{\mathcal{P}}_nf)$ is Cauchy in $D(A^{\frac{m}{2}})$ hence Cauchy in $W^{m,2}$ by Proposition \ref{prop all norm equivalences DA}. The limit in $W^{m,2}$ agrees with the limit in $L^2_{\sigma}$ which is $f$, and also belongs to $\bar{W}^{2,2}_{\alpha}$ as required. For the reverse inclusion, we use that $f \in W^{3,2} \cap \bar{W}^{2,2}_{\alpha}$ belongs to $D(A)$, so by Lemma \ref{lemma for passing stokes powers} it is sufficient to show that $Af \in D(A^{\frac{1}{2}}) = \bar{W}^{1,2}_{\sigma}$. This is clear however, as $\Delta f \in W^{1,2}$ so $\mathcal{P}\Delta f \in \bar{W}^{1,2}_{\sigma}$. 
\end{proof}

\begin{remark}
    This result is unavailable in the classical no-slip case, as there $D(A^{\frac{1}{2}})$ is equal to $W^{1,2}_{0} \cap L^2_{\sigma}$, and the Leray Projector destroys the zero trace property. 
\end{remark}

\begin{proposition} \label{big boy}
    $\bar{\mathcal{P}}_n$ is self-adjoint on $D(A^s)$. For large enough $n \in \N$ such that $\bar{\lambda}_n \geq 1$,  $f \in D(A^s)$ and $0 \leq r <s$, 
    $$\norm{(I - \bar{\mathcal{P}}_n)f}_{A^{r}}^2 \leq \frac{1}{\bar{\lambda}_n^{2(s-r)}}\norm{f}_{A^s}^2$$
    where $I$ is the identity operator.
\end{proposition}

\begin{proof}
    For the self-adjoint property, simply note that for $f,g \in D(A^s)$, $$\inner{\bar{\mathcal{P}}_nf}{g}_{A^s} = \inner{\sum_{k=1}^n\bar{\lambda}_k^s\inner{f}{\bar{a}_k}\bar{a}_k}{\sum_{j=1}^\infty\bar{\lambda}_k^s\inner{g}{\bar{a}_k}\bar{a}_k} = \sum_{k=1}^n\bar{\lambda}_k^{2s}\inner{f}{\bar{a}_k}\inner{g}{\bar{a}_k} = \inner{f}{\bar{\mathcal{P}}_ng}_{A^s}.$$
    Onto the approximation by $\bar{\mathcal{P}}_nf$ for $f \in D(A^s)$, we have that
    \begin{align*}
        \norm{(I - \bar{\mathcal{P}}_n)f}_{A^{r}}^2 &= \sum_{k=n+1}^\infty\bar{\lambda}_k^{2r}\inner{f}{\bar{a}_k}^2 = \frac{1}{\bar{\lambda}_n^{2(s-r)}}\sum_{k=n+1}^\infty\bar{\lambda}_k^{2r}\bar{\lambda}_n^{2(s-r)}\inner{f}{\bar{a}_k}^2\\ &\leq \frac{1}{\bar{\lambda}_n^{2(s-r)}}\sum_{k=n+1}^\infty\bar{\lambda}_k^{2s}\inner{f}{\bar{a}_k}^2 = \frac{1}{\bar{\lambda}_n^{2(s-r)}}\norm{(I - \bar{\mathcal{P}}_n)f}_{A^{s}}^2 \leq \frac{1}{\bar{\lambda}_n^{2(s-r)}}\norm{f}_{A^s}^2
    \end{align*}
    having used that $1 \leq \bar{\lambda}_n \leq \bar{\lambda}_k$ for $n \leq k$.
\end{proof}

\subsection{Noise Estimates} \label{subs noise estimates}

Throughout this subsection we assume that $\kappa \geq 0$ and $\alpha = 2\kappa$ as in the main theorems, and so that $\alpha \geq \kappa$ ensuring all results of Subsection \ref{subs frac} apply. We shall make frequent use of Corollary \ref{vorticity corollary}, where the second condition under these assumptions is that $\textnormal{curl}f = 0$ on $\partial \mathscr{O}$. 

\begin{lemma} \label{lemma for noise in right space}
    For $m \in \N$ even, let  $\xi_i \in L^2_{\sigma} \cap W^{m,2}_0 \cap W^{m+2,\infty}$ and $f \in  D(A^{\frac{m+2}{2}})$. Then both $\mathcal{P}B_i^2f$ and $ \mathcal{P}B_if$ belong to $D(A^{\frac{m}{2}})$. 
\end{lemma}

\begin{proof}
    Let $m = 2k$, $k \in \N$, and $f \in D(A^{k+1})$. We initially show that each term belongs to $D(A)$, and build up to $D(A^k)$. Our argument shall be for $\mathcal{P}B_i^2f$, with $ \mathcal{P}B_if$ following similarly. As $D(A) = \bar{W}^{2,2}_{2\kappa}$, then from Corollary \ref{vorticity corollary} it is sufficient to show that $\mathcal{P}B_i^2f \in W^{2,2} \cap \bar{W}^{1,2}_{\sigma}$ and that $\textnormal{curl}\left(\mathcal{P}B_i^2f \right) = 0$ on $\partial \mathscr{O}$. We first note that as $f \in D(A^{k+1})$ then $f \in W^{2(k+1),2}$ from Lemma \ref{lemma for inclusion of high A}, so at least $f \in W^{4,2}$. Thus $B_i^2f \in W^{2,2}$ and $\mathcal{P}B_i^2f \in W^{2,2} \cap \bar{W}^{1,2}_{\sigma}$. It remains to show that $\textnormal{curl}\left(\mathcal{P}B_i^2f \right) = 0$ on $\partial \mathscr{O}$, for which we refer to Lemma \ref{lemma for curl and P} and (\ref{curl of Bi}) to see that $\textnormal{curl}\left(\mathcal{P}B_i^2f \right) = \mathcal{L}_{\xi_i}^2(\textnormal{curl}f)$. As $\xi_i \in W^{2,2}_0$ then there exists a sequence of compactly supported functions $(\phi^n)$ approximating $\xi_i$ in $W^{2,2}$, which is sufficient to show that the compactly supported $\mathcal{L}_{\phi^n}\mathcal{L}_{\xi_i}(\textnormal{curl}f)$ approximates $\mathcal{L}_{\xi_i}^2(\textnormal{curl}f)$ in $W^{1,2}(\mathscr{O};\R)$ and hence $\mathcal{L}_{\xi_i}^2(\textnormal{curl}f)$ is of null trace. Therefore, $\mathcal{P}B_i^2f \in D(A)$.\\

With the base case established, let us now make the inductive hypothesis that for $1 \geq j <k$, $\mathcal{P}B_i^2f \in D(A^j)$ implies that $\mathcal{P}B_i^2f \in D(A^{j+1})$. If the hypothesis is true then the result is proven with $j=k-1$. Let us, therefore, assume that $\mathcal{P}B_i^2f \in D(A^j)$. By Lemma \ref{lemma for passing stokes powers}, we would prove that $\mathcal{P}B_i^2f \in D(A^{j+1})$ if we verify that $A^{j}\mathcal{P}B_i^2f \in D(A)$. Iterating the property that $A = A\mathcal{P}$ from the right side inwards, then $A^j\mathcal{P} = (-1)^j\mathcal{P}\Delta^j$ so $A^{j}\mathcal{P}B_i^2f = (-1)^j\mathcal{P}\Delta^jB_i^2f$. Showing that this belongs to $\bar{W}^{2,2}_{2\kappa}$ follows similarly to the base case, where we have the assumptions that $\xi_i \in W^{2k,2}_0 \cap W^{2k+2,\infty}$, and noting that the curl passes through the Laplacians.
\end{proof}

\begin{remark}
    You may observe that in fact, one only requires $f \in  D(A^{\frac{m+1}{2}})$ for $\mathcal{P}B_if \in D(A^{\frac{m}{2}})$. In practice $f$ will have much better regularity anyway, so we choose to maintain the flow of our work by not separating these cases. 
\end{remark}

\begin{corollary} \label{odd corollary}
    For $m \in \N$ odd, let  $\xi_i \in L^2_{\sigma} \cap W^{m+1,2}_0 \cap W^{m+3,\infty}$ and $f \in  D(A^{\frac{m+3}{2}})$. Then both $\mathcal{P}B_i^2f$ and $ \mathcal{P}B_if$ belong to $D(A^{\frac{m}{2}})$.
\end{corollary}

\begin{proof}
    We apply Lemma \ref{lemma for noise in right space} for the even $m +1$, and simply use that $D(A^{\frac{m+1}{2}}) \subseteq D(A^{\frac{m}{2}})$.
\end{proof}

Having established that the noise belongs to the suitable spaces, we can prove estimates in them.

\begin{proposition} \label{prop high order estimates 1}
We have the following high order estimates:
\begin{enumerate}
    \item \label{itemous one} Let $m \in \N$ be even. For any $\varepsilon > 0$, there exists a constant $c_{\varepsilon}$ such that for all $\xi_i \in L^2_{\sigma} \cap W^{m,2}_0 \cap W^{m+2,\infty}$ and $f \in D(A^{\frac{m+2}{2}})$, $$\inner{\mathcal{P}B_i^2f}{f}_{A^\frac{m}{2}} + \norm{\mathcal{P}B_if}_{A^\frac{m}{2}}^2 \leq c_{\varepsilon}\norm{\xi_i}_{W^{m+1,\infty}}^2\norm{f}_{A^{\frac{m}{2}}}^2 + \varepsilon\norm{\xi_i}_{W^{m+1,\infty}}^2\norm{f}_{A^{\frac{m+1}{2}}}^2;$$
    \item \label{itemous two} Let $m \in \N$ be odd. For any $\varepsilon > 0$, there exists a constant $c_{\varepsilon}$ such that for all $\xi_i \in L^2_{\sigma} \cap W^{m+1,2}_0 \cap W^{m+3,\infty}$ and $f \in  D(A^{\frac{m+3}{2}})$, $$\inner{\mathcal{P}B_i^2f}{f}_{A^\frac{m}{2}} + \norm{\mathcal{P}B_if}_{A^\frac{m}{2}}^2 \leq c_{\varepsilon}\norm{\xi_i}_{W^{m+2,\infty}}^2\norm{f}_{A^{\frac{m}{2}}}^2 + \varepsilon\norm{\xi_i}_{W^{m+2,\infty}}^2\norm{f}_{A^{\frac{m+1}{2}}}^2.$$
\end{enumerate}
\end{proposition}

\begin{proof}
Beginning with item \ref{itemous one}, let $m = 2k$, $k \in \N$, and $f \in D(A^{k+1})$. We must control \begin{equation} \label{theonetwo}
    \inner{A^k\mathcal{P}B_i^2f}{A^kf} + \inner{A^k\mathcal{P}B_if}{A^k\mathcal{P}B_if}.
\end{equation}
Our idea is to use the representation $A^k\mathcal{P} = (-1)^k\mathcal{P}\Delta^k$ seen in the proof of Lemma \ref{lemma for noise in right space}, commute $B_i$ with $\Delta$ through (\ref{commutator bound higher order}), and then use the cancellation arising from $B_i^*$. Considering the first term, we write \begin{align*}A^k\mathcal{P}B_i^2f &= (-1)^k\mathcal{P}\Delta^kB_i^2f\\ &= (-1)^k\mathcal{P}\Delta^{k-1}B_i\Delta B_if + (-1)^k\mathcal{P}\Delta^{k-1}[\Delta,B_i]B_if\\
&= (-1)^k\mathcal{P}\Delta^{k-2}B_i\Delta^2 B_if + (-1)^k\mathcal{P}\Delta^{k-2}[\Delta,B_i]\Delta B_if +(-1)^k\mathcal{P}\Delta^{k-1}[\Delta,B_i]B_if\\
&= (-1)^k\mathcal{P}B_i\Delta^{k} B_if + (-1)^k\mathcal{P}\sum_{j=1}^k\Delta^{k-j}[\Delta,B_i]\Delta^{j-1}B_if\\
&= \mathcal{P}B_iA^{k} B_if + (-1)^k\mathcal{P}\sum_{j=1}^k\Delta^{k-j}[\Delta,B_i]\Delta^{j-1}B_if
\end{align*}
where in the last line we have used that $\mathcal{P}B_i = \mathcal{P}B_i\mathcal{P}$, linearity of $\mathcal{P}$ and $B_i$, and that $(-1)^k\mathcal{P}\Delta^k = A^k$. Identically, pertaining to the second term of (\ref{theonetwo}),
\begin{equation} \label{le pertinence} A^k\mathcal{P}B_if =  \mathcal{P}B_iA^{k}f + (-1)^k\mathcal{P}\sum_{j=1}^k\Delta^{k-j}[\Delta,B_i]\Delta^{j-1}f.\end{equation}
Therefore, we can rewrite (\ref{theonetwo}) as
\begin{align*}
    (\ref{theonetwo}) & = \inner{\mathcal{P}B_iA^{k} B_if + (-1)^k\mathcal{P}\sum_{j=1}^k\Delta^{k-j}[\Delta,B_i]\Delta^{j-1}B_if}{A^kf}\\ & \qquad \qquad \qquad + \inner{A^k\mathcal{P}B_if}{\mathcal{P}B_iA^{k}f + (-1)^k\mathcal{P}\sum_{j=1}^k\Delta^{k-j}[\Delta,B_i]\Delta^{j-1}f}\\
    &= \inner{B_iA^{k} B_if}{A^kf} + \inner{(-1)^k\mathcal{P}\sum_{j=1}^k\Delta^{k-j}[\Delta,B_i]\Delta^{j-1}B_if}{A^kf}\\ & \qquad \qquad \qquad
    + \inner{A^kB_if}{B_iA^{k}f}  + \inner{A^kB_if}{(-1)^k\mathcal{P}\sum_{j=1}^k\Delta^{k-j}[\Delta,B_i]\Delta^{j-1}f}\\
    & = \inner{A^{k} B_if}{(B_i^* + B_i)A^kf} + \inner{(-1)^k\mathcal{P}\sum_{j=1}^k\Delta^{k-j}[\Delta,B_i]\Delta^{j-1}B_if}{A^kf}\\ & \qquad \qquad \qquad
    + \inner{A^kB_if}{(-1)^k\mathcal{P}\sum_{j=1}^k\Delta^{k-j}[\Delta,B_i]\Delta^{j-1}f}
\end{align*}
having used that $\mathcal{P}$ is self-adjoint and that $\mathcal{P}A^k = A^k = A^k\mathcal{P}$. The remainder of the proof is simply a combination of Cauchy-Schwarz and Young's Inequality, having achieved cancellation of a derivative in all terms. We note that the proof of $\norm{A^kg} \leq c\norm{g}_{W^{2k,2}}$ for $g \in D(A^k)$ of Proposition \ref{prop all norm equivalences DA} is immediate also for general $g \in W^{2k,2}$. In the first term,
\begin{align}
    \nonumber \inner{A^{k} B_if}{(B_i^* + B_i)A^kf} &= \inner{A^{k} B_if}{(\mathcal{T}_{\xi_i}^* + \mathcal{T}_{\xi_i})A^kf}\\ \nonumber
    &\leq \norm{A^{k} B_if}\norm{(\mathcal{T}_{\xi_i}^* + \mathcal{T}_{\xi_i})A^kf}\\ \nonumber
    &\leq c\norm{B_if}_{W^{2k,2}}\norm{\xi_i}_{W^{1,\infty}}\norm{A^kf}\\  \nonumber
    &\leq c\norm{\xi_i}_{W^{2k+1,\infty}}\norm{f}_{W^{2k+1,2}}\norm{\xi_i}_{W^{1,\infty}}\norm{A^kf}\\
    &\leq c_{\varepsilon}\norm{\xi_i}_{W^{2k+1,\infty}}^2\norm{f}_{A^k}^2 + \varepsilon\norm{\xi_i}_{W^{2k+1,\infty}}^2\norm{f}_{A^{\frac{2k+1}{2}}}^2. \label{gfhj}
\end{align}
Approaching the next term, using (\ref{commutator bound higher order}), 
\begin{align}
     \nonumber \norm{(-1)^k\mathcal{P}\sum_{j=1}^k\Delta^{k-j}[\Delta,B_i]\Delta^{j-1}B_if} &\leq c\sum_{j=1}^k\norm{\Delta^{k-j}[\Delta,B_i]\Delta^{j-1}B_if}\\  \nonumber
    &\leq c\sum_{j=1}^k\norm{[\Delta,B_i]\Delta^{j-1}B_if}_{W^{2(k-j)}}\\  \nonumber
    &\leq c\sum_{j=1}^k\norm{\xi_i}_{W^{2(k-j)+3,\infty}}\norm{\Delta^{j-1}B_if}_{W^{2(k-j)+2,2}}\\  \nonumber
     &\leq c\sum_{j=1}^k\norm{\xi_i}_{W^{2(k-j)+3,\infty}}\norm{B_if}_{W^{2k,2}}\\  \nonumber
     &\leq c\sum_{j=1}^k\norm{\xi_i}_{W^{2(k-j)+3,\infty}}\norm{\xi_i}_{W^{2k+1,\infty}}\norm{f}_{W^{2k+1,2}}\\ \label{exammarks}
     &\leq c\norm{\xi_i}_{W^{2k+1,\infty}}^2\norm{f}_{W^{2k+1,2}}.
\end{align}
Therefore
\begin{align*}\inner{(-1)^k\mathcal{P}\sum_{j=1}^k\Delta^{k-j}[\Delta,B_i]\Delta^{j-1}B_if}{A^kf} &\leq c\norm{\xi_i}_{W^{2k+1,\infty}}^2\norm{f}_{W^{2k+1,2}}\norm{A^kf}\\
&\leq c_{\varepsilon}\norm{\xi_i}_{W^{2k+1,\infty}}^2\norm{f}_{A^k}^2 + \varepsilon\norm{\xi_i}_{W^{2k+1,\infty}}^2\norm{f}_{A^{\frac{2k+1}{2}}}^2
\end{align*}
where the same arguments generate the same control on the final term, concluding the proof of item \ref{itemous one}. Item \ref{itemous two} is unsurprisingly alike, so letting $m=2k-1$ and $f \in D(A^{k+1})$, we make the initial observation with Lemma \ref{lemma for passing stokes powers} that
$$\inner{A^{\frac{2k-1}{2}}\mathcal{P}B_i^2f}{A^{\frac{2k-1}{2}}f} + \inner{A^{\frac{2k-1}{2}}\mathcal{P}B_if}{A^{\frac{2k-1}{2}}\mathcal{P}B_if} = \inner{A^{k-1}\mathcal{P}B_i^2f}{A^{k}f} + \inner{A^{k-1}\mathcal{P}B_if}{A^{k}\mathcal{P}B_if}.$$
With the same approach that was used for (\ref{theonetwo}), we represent this as
\begin{align}
   \nonumber \inner{A^{k-1}\mathcal{P}B_i^2f}{A^{k}f} &+ \inner{A^{k-1}\mathcal{P}B_if}{A^{k}\mathcal{P}B_if}\\ \nonumber &= \inner{A^{k-1} B_if}{(B_i^* + B_i)A^kf} + \inner{(-1)^{k-1}\mathcal{P}\sum_{j=1}^{k-1}\Delta^{k-1-j}[\Delta,B_i]\Delta^{j-1}B_if}{A^kf}\\ & \qquad \qquad \qquad
    + \inner{A^{k-1}B_if}{(-1)^k\mathcal{P}\sum_{j=1}^k\Delta^{k-j}[\Delta,B_i]\Delta^{j-1}f}. \label{nostarforyou}
\end{align}
Similarly to (\ref{gfhj}),
\begin{align*}
    \inner{A^{k-1} B_if}{(B_i^* + B_i)A^kf} &\leq c\norm{B_if}_{W^{2(k-1),2}}\norm{(\mathcal{T}_{\xi_i}^* + \mathcal{T}_{\xi_i})A^kf}\\ &\leq c\norm{\xi_i}_{W^{2k-1,\infty}}^2\norm{f}_{W^{2k-1,2}}\norm{f}_{W^{2k,2}}\\
    &\leq c_{\varepsilon}\norm{\xi_i}_{W^{2k-1,\infty}}^2\norm{f}_{A^{\frac{2k-1}{2}}}^2 + \varepsilon\norm{\xi_i}_{W^{2k-1,\infty}}^2\norm{f}_{A^{k}}^2
\end{align*}
and comparing to (\ref{exammarks}),
\begin{align} \nonumber
    \norm{(-1)^{k-1}\mathcal{P}\sum_{j=1}^{k-1}\Delta^{k-1-j}[\Delta,B_i]\Delta^{j-1}B_if} \leq c\norm{\xi_i}_{W^{2k-1,\infty}}^2\norm{f}_{W^{2k-1,2}}
\end{align}
as well as 
\begin{align} \label{just fits}
    \norm{(-1)^k\mathcal{P}\sum_{j=1}^k\Delta^{k-j}[\Delta,B_i]\Delta^{j-1}f} \leq c\sum_{j=1}^k\norm{\xi_i}_{W^{2(k-j)+3,\infty}}\norm{f}_{W^{2k,2}} \leq c\norm{\xi_i}_{W^{2k+1,\infty}}\norm{f}_{W^{2k,2}}.
\end{align}
It is now straightforwards to slot these bounds into (\ref{nostarforyou}) with Cauchy-Scwharz and Young's Inequality, concluding the proof in the same manner as item \ref{itemous one}.
    
\end{proof}

Practically for energy estimates of (\ref{projected Ito Salt}), Proposition \ref{prop high order estimates 1} is used for the sum between the It\^{o}-Stratonovich corrector and the quadratic variation of the stochastic integral. For the stochastic integral itself, upon applying the Burkholder-Davis-Gundy Inequality, we will use Proposition \ref{high order estimates 2}.

\begin{proposition} \label{high order estimates 2}
We have the following high order estimates:
\begin{enumerate}
    \item \label{itemous one prime} Let $m \in \N$ be even. There exists a constant $c$ such that for all $\xi_i \in L^2_{\sigma} \cap W^{m,2}_0 \cap W^{m+2,\infty}$ and $f \in D(A^{\frac{m+2}{2}})$, $$\inner{\mathcal{P}B_if}{f}_{A^\frac{m}{2}}^2 \leq c\norm{\xi_i}_{W^{m+1,\infty}}^2\norm{f}_{A^{\frac{m}{2}}}^4;$$
    \item \label{itemous two prime} Let $m \in \N$ be odd. There exists a constant $c$ such that for all $\xi_i \in L^2_{\sigma} \cap W^{m+1,2}_0 \cap W^{m+3,\infty}$ and $f \in  D(A^{\frac{m+3}{2}})$, $$\inner{\mathcal{P}B_if}{f}_{A^\frac{m}{2}}^2 \leq c\norm{\xi_i}_{W^{m+1,\infty}}^2\norm{f}_{A^{\frac{m}{2}}}^4.$$
\end{enumerate}
\end{proposition}


\begin{proof}
Once more we begin with the even case $m=2k$. We must control $$\inner{A^k \mathcal{P}B_if}{A^kf}^2$$
which we approach similarly to Proposition \ref{prop high order estimates 1} in terms of commutator arguments. Recalling (\ref{le pertinence}), followed by using that $\mathcal{P}$ is self-adjoint and the bounds (\ref{finalboundinderivativeproof}), (\ref{just fits}), we obtain that
\begin{align*}
    \inner{A^k \mathcal{P}B_if}{A^kf}^2 &= \left(\inner{\mathcal{P}B_iA^{k}f}{A^kf} + \inner{(-1)^k\mathcal{P}\sum_{j=1}^k\Delta^{k-j}[\Delta,B_i]\Delta^{j-1}f}{A^kf}\right)^2\\
    &=  \left(\inner{B_iA^{k}f}{A^kf} + \inner{(-1)^k\mathcal{P}\sum_{j=1}^k\Delta^{k-j}[\Delta,B_i]\Delta^{j-1}f}{A^kf}\right)^2\\
    &\leq c\norm{\xi_i}^2_{W^{2k+1,\infty}}\norm{f}_{A^k}^4.
\end{align*}
Moving on to odd $m = 2k-1$, our starting expression is $$\inner{A^\frac{m}{2} \mathcal{P}B_if}{A^{\frac{m}{2}}f}^2 = \inner{A^k \mathcal{P}B_if}{A^{k-1}f}^2.$$
We now look to commute only the first $k-1$ powers of $A$ with $\mathcal{P}B_i$, leading to
\begin{align*}
    \inner{A^k \mathcal{P}B_if}{A^{k-1}f}^2 &= \left(\inner{A \mathcal{P}B_iA^{k-1}f}{A^{k-1}f} + \inner{A(-1)^{k-1}\mathcal{P}\sum_{j=1}^{k-1}\Delta^{k-1-j}[\Delta,B_i]\Delta^{j-1}f}{A^{k-1}f}\right)^2\\
    &\leq 2\inner{\Delta B_iA^{k-1}f}{A^{k-1}f}^2 + 2\inner{A(-1)^{k-1}\mathcal{P}\sum_{j=1}^{k-1}\Delta^{k-1-j}[\Delta,B_i]\Delta^{j-1}f}{A^{k-1}f}^2.
\end{align*}
As the second term appears more familiar we start there, though we must be precise as there are one too many derivatives in the left side of the inner product. We handle this by observing that $A^{k-1}\mathcal{P}B_if \in D(A)$ given that $\mathcal{P}B_if \in D(A^{k})$, and as $\mathcal{P}B_iA^{k-1}f \in D(A)$ then so too is the difference $(-1)^{k-1}\mathcal{P}\sum_{j=1}^{k-1}\Delta^{k-1-j}[\Delta,B_i]\Delta^{j-1}f$. Therefore,
$$ \inner{A(-1)^{k-1}\mathcal{P}\sum_{j=1}^{k-1}\Delta^{k-1-j}[\Delta,B_i]\Delta^{j-1}f}{A^{k-1}f}^2 = \inner{A^{\frac{1}{2}}(-1)^{k-1}\mathcal{P}\sum_{j=1}^{k-1}\Delta^{k-1-j}[\Delta,B_i]\Delta^{j-1}f}{A^{k-\frac{1}{2}}f}^2.$$
The desired bound now comes out of Cauchy-Schwarz,
\begin{align*}
    \norm{A^{\frac{1}{2}}(-1)^{k-1}\mathcal{P}\sum_{j=1}^{k-1}\Delta^{k-1-j}[\Delta,B_i]\Delta^{j-1}f} &\leq c\sum_{j=1}^{k-1}\norm{\Delta^{k-1-j}[\Delta,B_i]\Delta^{j-1}f}_1\\ &\leq c\sum_{j=1}^{k-1}\norm{[\Delta,B_i]\Delta^{j-1}f}_{W^{2(k-j)-1,2}}\\
    &\leq c\norm{\xi_i}_{W^{2k,\infty}}\norm{f}_{W^{2k-1,2}}
\end{align*}
again using (\ref{commutator bound higher order}). To conclude the proof it only remains to control \begin{equation} \label{comptroller} \inner{\Delta B_iA^{k-1}f}{A^{k-1}f}^2.\end{equation}
For this we can use that $\xi_i \in W^{m+1,2}_0$ to carry out an integration by parts with null boundary term, reducing (\ref{comptroller}) to $$\inner{B_iA^{k-1}f}{A^{k-1}f}_1^2.$$ This is now controlled by (\ref{finalboundinderivativeproof}) as required. 

\end{proof}

\subsection{Strong Solutions of the Fully Hyperdissipative Equation} \label{subs reg of hyper}

We now prove a first existence and uniqueness result of this paper, achieved as a swift application of the variational framework developed in [\cite{goodair2024weak}]. This result will be improved upon in Subsection \ref{subs improved reg} by reducing the required hyperdissipation, though its proof relies upon the solutions given here. 

\begin{definition}
    Two processes $u$ and $v$ are said to be indistinguishable if  $$ \mathbbm{P}\left(\left\{\omega \in \Omega: u_t(\omega) = v_t(\omega) \quad \forall t \in [0,T] \right\}\right) = 1.$$
\end{definition}

\begin{proposition} \label{2d strong  navier hyperdis1}
    Let $\alpha = 2\kappa$ and $\kappa \geq 0$. For a given $m \in \N$ let $u_0: \Omega \rightarrow D(A^{\frac{m}{2}})$ be $\mathcal{F}_0-$measurable and $\xi_i \in L^2_{\sigma} \cap W^{m+1,2}_0 \cap W^{m+3,\infty}$ such that $\sum_{i=1}^\infty \norm{\xi_i}_{W^{m+2,\infty}}^2 < \infty$. Then there exists a progressively measurable process $u$ in $D(A^m)$ such that for $\mathbbm{P}-a.e.$ $\omega$, $u_{\cdot}(\omega) \in C\left([0,T];D(A^{\frac{m}{2}})\right) \cap L^2\left([0,T];D(A^m)\right)$ and \begin{equation} \nonumber
    u_t = u_0 - \int_0^t\mathcal{P}\mathcal{L}_{u_s}u_s\ ds - \nu\int_0^t A^m u_s\, ds + \frac{1}{2}\int_0^t\sum_{i=1}^\infty \mathcal{P}B_i^2u_s ds - \int_0^t \mathcal{P}Bu_s d\mathcal{W}_s 
\end{equation} holds $\mathbbm{P}-a.s.$ in $L^2_{\sigma}$ for all $t \in [0,T]$. Moreover if $v$ is any other such process, then $u$ and $v$ are indistinguishable.
\end{proposition}

This result was proven in [\cite{goodair2023navier}] for the case $m=1$, 
as an application of Theorem \ref{theorem for strong existence} presented in the appendix. Our proof of Proposition \ref{2d strong  navier hyperdis1} again comes as an application of Theorem \ref{theorem for strong existence}, so we present the details of this application here for $m \geq 2$. Our equation verifies the functional framework of Subsection \ref{subs functional framework} for the spaces \begin{align*}
    V:= D(A^m), \qquad H:= D(A^{\frac{m}{2}}), \qquad U:= L^2_{\sigma}.
\end{align*}
We note that $D(A^{\frac{m}{2}}) \hookrightarrow \bar{W}^{1,2}_{\sigma}$ which is compactly embedded into $L^2_{\sigma}$. We use the system of eigenfunctions of the Stokes Operator from Lemma \ref{eigenfunctions for navier}, whose properties are explored in Subsection \ref{functional framework subsection}. Note that the nonlinear term maps from $D(A^{\frac{m}{2}})$ into $L^2_{\sigma}$ (recall once more that $m \geq 2$), so its representation as an element of $\left(D(A^{\frac{m}{2}})\right)^*$ is clear. For the hyperdissipative term, we have that $$\inner{A^mf}{g}_{\left(D(A^{\frac{m}{2}})\right)^* \times D(A^{\frac{m}{2}})} = \inner{A^{\frac{m}{2}}f}{A^{\frac{m}{2}}g}.$$
The verification of Assumption Sets 2 and 3 pose no additional difficulties to the case $m=1$, as these are $L^2_{\sigma}$ based estimates. For Assumption Set 3, we may take $\bar{H} = H = D(A^{\frac{m}{2}})$, recalling that the noise maps into the right space from Lemma \ref{lemma for noise in right space}. The approximation (\ref{mu2}) was shown in Proposition \ref{big boy}. Assumption \ref{newfacilitator} sees little difference to the $m=1$ case, so we conclude by addressing Assumption \ref{assumptions for uniform bounds2}. The necessary noise estimates were the content of Propositions \ref{prop high order estimates 1} and \ref{high order estimates 2}, using that $\bar{\mathcal{P}}_n$ is self-adjoint in this space, and taking $$\varepsilon = \frac{\nu}{2\sum_{i=1}^\infty \norm{\xi_i}_{W^{m+2,\infty}}^2}$$ in Proposition \ref{prop high order estimates 1}. Of course the hyperdissipative term gives us simply $$\inner{A^m\phi^n}{\phi^n}_{A^{\frac{m}{2}}} = \norm{\phi^n}_{A^m}^2 $$ from Lemma \ref{lemma for passing stokes powers}. In the nonlinear term we do not have that $\mathcal{P}\mathcal{L}_{\phi^n}\phi^n \in D(A^{\frac{m}{2}})$ so we cannot simply pass the $\bar{\mathcal{P}}_n$ to the other side, or use some uniform boundedness property in this space.\footnote{This fact is the reason why we require hyperdissipativity at all.} To control this term, we instead observe that $$\inner{\bar{\mathcal{P}}_n\mathcal{P}\mathcal{L}_{\phi^n}\phi^n}{\phi^n}_{A^{\frac{m}{2}}} = \inner{A^{\frac{1}{2}}\bar{\mathcal{P}}_n\mathcal{P}\mathcal{L}_{\phi^n}\phi^n}{A^{\frac{2m-1}{2}}\phi^n} \leq \norm{\mathcal{L}_{\phi^n}\phi^n}_{W^{1,2}}\norm{A^{\frac{2m-1}{2}}\phi^n}$$
using that $\mathcal{P}\mathcal{L}_{\phi^n}\phi^n \in D(A^{\frac{1}{2}})$, the equivalence of the norms and the boundedness of the Leray Projector. The control $$\norm{\mathcal{L}_{\phi^n}\phi^n}_{W^{1,2}} \leq c\norm{\phi^n}_{W^{2,2}}^2$$
is classical. Using the possibly coarse bound $\norm{\phi^n}_{W^{2,2}} \leq c\norm{\phi^n}_{W^{m,2}}$, the equivalence of the given norms with their fractional Stokes counterparts as well as Young's Inequality, we ultimately have that
\begin{align} \nonumber
\inner{\bar{\mathcal{P}}_n\mathcal{P}\mathcal{L}_{\phi^n}\phi^n}{\phi^n}_{A^{\frac{m}{2}}}
\leq c\norm{\phi^n}_{A^{\frac{m}{2}}}^4 + \frac{\nu}{2}\norm{\phi^n}_{A^m}^2.
\end{align}
In fact we have a sharper bound,
\begin{align}
    \inner{\bar{\mathcal{P}}_n\mathcal{P}\mathcal{L}_{\phi^n}\phi^n}{\phi^n}_{A^{\frac{m}{2}}} \leq c\norm{\phi^n}_{A}^2\norm{\phi^n}_{A^{\frac{m}{2}}}^2 + \frac{\nu}{2}\norm{\phi^n}_{A^{\frac{2m-1}{2}}}^2 \label{lecombina 1}
\end{align}
 which will prove relevant in Theorem \ref{2d strong  navier hyperdis main}. This completely verifies an application of Theorem \ref{theorem for strong existence} which only leaves the pathwise continuity in Proposition \ref{2d strong  navier hyperdis1} to be shown, for which we apply Lemma \ref{continuity lemma 2} with the bilinear form $$\inner{f}{\phi}_{U \times V}:= \inner{f}{A^m\phi}.$$

\section{Existence of Strong Solutions} \label{section existence of strong solutions}

In this section we state and prove the main existence and uniqueness results of this work. Subsection \ref{subs improved reg} is concerned with reducing the dissipation required in Proposition \ref{2d strong  navier hyperdis1}, leading to our key existence result for the It\^{o} form (\ref{projected Ito Salt hyperdiss}). In Subsection \ref{subs strong strat}, the solutions obtained in Subsection \ref{subs improved reg} are shown to be strong solutions of the Stratonovich form (\ref{projected strato Salt hyperdiss}). We use the abstract It\^{o}-Stratonovich conversion of [\cite{goodair2022stochastic}], stated in the appendix, Section \ref{Appendix II}. Particular attention is drawn to the case of no hyperdissipation, obtaining genuine strong solutions of (\ref{projected strato Salt}). 

\subsection{Improved Regularity with Reduced Dissipation} \label{subs improved reg}

The main result of this subsection is the following.

\begin{theorem} \label{2d strong  navier hyperdis main}
    Let $\alpha = 2\kappa$ and $\kappa \geq 0$. For a given $m \in \N$, $m \geq 2$, let $u_0: \Omega \rightarrow D(A^{\frac{m}{2}})$ be $\mathcal{F}_0-$measurable, $\xi_i \in L^2_{\sigma} \cap W^{m+1,2}_0 \cap W^{m+3,\infty}$ such that $\sum_{i=1}^\infty \norm{\xi_i}_{W^{m+2,\infty}}^2 < \infty$. Then there exists a progressively measurable process $u$ in $D(A^{\frac{2m-1}{2}})$ such that for $\mathbbm{P}-a.e.$ $\omega$, $u_{\cdot}(\omega) \in C\left([0,T];D(A^{\frac{m}{2}})\right) \cap L^2\left([0,T];D(A^{\frac{2m-1}{2}})\right)$ and \begin{equation} \label{numbertime}
    u_t = u_0 - \int_0^t\mathcal{P}\mathcal{L}_{u_s}u_s\ ds - \nu\int_0^t A^{m-1} u_s\, ds + \frac{1}{2}\int_0^t\sum_{i=1}^\infty \mathcal{P}B_i^2u_s ds - \int_0^t \mathcal{P}Bu_s d\mathcal{W}_s 
\end{equation} holds $\mathbbm{P}-a.s.$ in $\bar{W}^{1,2}_{\sigma}$ for all $t \in [0,T]$. Moreover if $v$ is any other such process, then $u$ and $v$ are indistinguishable.
\end{theorem}

Our idea is to apply Theorem \ref{theorem1} with the spaces $$\mathscr{V}:= D(A^{\frac{2m-1}{2}}), \qquad \mathscr{H}:= D(A^{\frac{m}{2}}), \qquad \mathscr{U}:= D(A^{\frac{1}{2}})$$ to the equation (\ref{numbertime}), verifying the existence of a maximal solution until blow-up in a norm which is known to remain finite on $[0,T]$ from Proposition \ref{2d strong  navier hyperdis1}. To make our use of Proposition \ref{2d strong  navier hyperdis1} precise we fix the assumptions of Theorem \ref{2d strong  navier hyperdis main} and write down the application of this proposition to equation (\ref{numbertime}) as the following lemma.

\begin{lemma} \label{for the unique u}
    There exists a progressively measurable process $u$ in $D(A^{m-1})$ such that for $\mathbbm{P}-a.e.$ $\omega$, $u_{\cdot}(\omega) \in C\left([0,T];D(A^{\frac{m-1}{2}})\right) \cap L^2\left([0,T];D(A^{m-1})\right)$ and (\ref{numbertime}) holds $\mathbbm{P}-a.s.$ in $L^2_{\sigma}$ for all $t \in [0,T]$. Moreover if $v$ is any other such process, then $u$ and $v$ are indistinguishable.
\end{lemma}

Our use of Theorem \ref{theorem1} can be seen as upgrading the regularity on the unique process $u$ specified in Lemma \ref{for the unique u} to that required in Theorem \ref{2d strong  navier hyperdis main}. To verify the assumptions of Theorem \ref{theorem1} we note that in the case $m=2$ these assumptions were completely verified \textit{on the torus} in [\cite{goodair20233d}] Section 3. Given the work done in Subsections \ref{subs frac} and \ref{subs noise estimates} there are no additional details needed to verify the assumptions in our case. Indeed even for larger $m$ it is only Assumption \ref{uniform assumpt} that requires particular attention, as the space $\mathscr{U}$ remains the same. For the nonlinear term we use (\ref{lecombina 1}), and for the Stokes Operator we simply note that $$\inner{A^{m-1}\phi^n}{\phi^n}_{A^{\frac{m}{2}}} = \norm{\phi^n}_{A^{\frac{2m-1}{2}}}^2.$$ Estimates on the noise again come from Subsection \ref{subs noise estimates}, and the bilinear form takes the representation $$\inner{f}{\psi}_{\mathscr{U} \times \mathscr{V}} := \inner{A^{\frac{1}{2}}f}{A^{\frac{2m-1}{2}}\psi} = \inner{f}{\psi}_{A^{\frac{m}{2}}} $$ for $f \in D(A^{\frac{m}{2}})$. We omit further details and apply Theorem \ref{theorem1} in this context, obtaining a maximal solution of (\ref{numbertime}) in the sense of Definition \ref{V valued maximal definition} which in light of uniqueness must agree with the $u$ specified in Lemma \ref{for the unique u} on its time of existence. We deduce the following.

\begin{lemma} \label{lemma for higher order}
    Let $u$ be the unique process specified in Lemma \ref{for the unique u}. Then if \begin{equation} \label{big if}\sup_{r \in [0,T)}\norm{u_r}_{A^{\frac{1}{2}}}^2 + \int_0^{T}\norm{u_r}_{ A^{\frac{m}{2}}}^2dr < \infty \end{equation} $\mathbbm{P}-a.s.$, we have that $u$ is progressively measurable process in $D(A^{\frac{2m-1}{2}})$ and for $\mathbbm{P}-a.e.$ $\omega$, $u_{\cdot}(\omega) \in C\left([0,T];D(A^{\frac{m}{2}})\right) \cap L^2\left([0,T];D(A^{\frac{2m-1}{2}})\right)$.
\end{lemma}

\begin{proof}
    We apply Theorem \ref{theorem1} as described, where local strong solutions in this sense are progressively measurable in $D(A^{\frac{2m-1}{2}})$ and belong pathwise to $C\left([0,\tau];D(A^{\frac{m}{2}})\right) \cap L^2\left([0,\tau];D(A^{\frac{2m-1}{2}})\right)$. The result comes from the second assertion of Theorem \ref{theorem1} by choosing $\tau$ as simply $T$. 
    \end{proof}

Moreover, to prove Theorem \ref{2d strong  navier hyperdis main} it is sufficient to verify (\ref{big if}). This is, however, clear from the known regularity of Lemma \ref{for the unique u}, appreciating that as $m \geq 2$ then $m -1 \geq \frac{m}{2}$. The proof of Theorem \ref{2d strong  navier hyperdis main} is complete.

\subsection{Strong Solutions of the Stratonovich Equation} \label{subs strong strat}

The key result of this subsection is the following.

\begin{theorem} \label{Stratotheorem1}
    Let $u$ be the unique \footnote{By unique we mean `up to indistinguishability'.} process specified in Theorem \ref{2d strong  navier hyperdis main}. Then $u$ satisfies the identity 
$$  u_t = u_0 - \int_0^t\mathcal{P}\mathcal{L}_{u_s}u_s\ ds - \nu\int_0^t A^{m-1} u_s\, ds  - \int_0^t \mathcal{P}Bu_s \circ d\mathcal{W}_s$$
$\mathbbm{P}-a.s.$ in $L^2_{\sigma}$ for all $t \in [0,T]$.
\end{theorem}

We emphasise the case where $m=2$ and state the full result below, using the equivalence of spaces given in Proposition \ref{prop norm equivalence for W22} and Lemma \ref{lemma for inclusion of high A}. 

\begin{theorem} \label{Stratotheorem2}
    Let $\alpha = 2\kappa$, $\kappa \geq 0$, $u_0: \Omega \rightarrow \bar{W}^{2,2}_{\alpha}$ be $\mathcal{F}_0-$measurable and $\xi_i \in L^2_{\sigma} \cap W^{3,2}_0 \cap W^{5,\infty}$ such that $\sum_{i=1}^\infty \norm{\xi_i}_{W^{4,\infty}}^2 < \infty$. Then there exists a progressively measurable process $u$ in $W^{3,2} \cap \bar{W}^{2,2}_{\alpha}$ such that for $\mathbbm{P}-a.e.$ $\omega$, $u_{\cdot}(\omega) \in C\left([0,T];\bar{W}^{2,2}_{\alpha})\right) \cap L^2\left([0,T];W^{3,2}\right)$ and $u$ satisfies the identity 
$$  u_t = u_0 - \int_0^t\mathcal{P}\mathcal{L}_{u_s}u_s\ ds - \nu\int_0^t A u_s\, ds  - \int_0^t \mathcal{P}Bu_s \circ d\mathcal{W}_s$$
$\mathbbm{P}-a.s.$ in $L^2_{\sigma}$ for all $t \in [0,T]$.
\end{theorem}

As stated Theorem \ref{Stratotheorem2} follows as a particular case of Theorem \ref{Stratotheorem1}, hence it is sufficient to prove Theorem \ref{Stratotheorem1} alone. This follows directly from Theorem \ref{theorem for ito strat conversion} for the spaces
\begin{align*}
    V:= D(A^{\frac{2m-1}{2}}),  \quad H:= D(A^{\frac{m}{2}}), \quad
    U:= D(A^{\frac{1}{2}}), \quad  X:= L^2_{\sigma}.
\end{align*}
The assumptions of Theorem \ref{theorem for ito strat conversion} are immediate in light of Subsection \ref{subs noise estimates}. We conclude the proof here. 

\subsection{The Torus}

For completeness we briefly address the equation (\ref{projected strato Salt}) posed over the $N-$ dimensional torus $\T$ for $N=$ 2 or 3 dimensions, where hyperdissipation is not required. For the functional analytic framework we recall that any function $f \in L^2(\T;\R^N)$ admits the representation \begin{equation} \label{fourier rep}f(x) = \sum_{k \in \mathbb{Z}^N}f_ke^{ik\cdot x}\end{equation} whereby each $f_k \in \mathbb{C}^N$ is such that $f_k = \overbar{f_{-k}}$ and the infinite sum is defined as a limit in $L^2(\T;\R^N)$, see e.g. [\cite{robinson2016three}] Subsection 1.5 for details. In this setting we can make the following definition.

\begin{definition}
We define $L^2_{\sigma}$ as the subset of $L^2(\T;\R^N)$ of zero-mean functions $f$ whereby for all $k \in \mathbbm{Z}^N$, $k \cdot f_k = 0$ with $f_k$ as in (\ref{fourier rep}). For general $m \in \N$ we introduce $W^{m,2}_{\sigma}$ as the intersection of $W^{m,2}(\T;\R^N)$ respectively with $L^2_{\sigma}$.
\end{definition}

Note that the dimensionality $N$ is not explicitly included in the spaces, but will be made clear from context. In either case the spaces $D(A^s)$ can be introduced exactly as in Subsection \ref{subs frac}, with characterisation $D(A^{\frac{m}{2}}) = W^{m,2}_{\sigma}$. Further details can be found in [\cite{robinson2016three}] Exercises 2.12, 2.13 and the discussion in Subsection 2.3. The higher order regularity of solutions via an iterated application of Theorem \ref{theorem1} was shown in [\cite{goodair2024improved}] Theorem 4.3, for a Lipschitz noise. With the noise estimates of Subsection \ref{subs noise estimates}, clearly holding on the torus as well, it is straightforwards to apply the same procedure for the SALT noise considered in this paper. Combined with the It\^{o}-Stratonovich conversion as in Theorem \ref{Stratotheorem2}, we have the following.

\begin{proposition} \label{navier strong existence}
     For $m \geq 2$ let $u_0: \Omega \rightarrow W^{m,2}_{\sigma}$ be $\mathcal{F}_0-$measurable and $(u,\tau)$ be a local strong solution of the equation \begin{equation} \nonumber
    u_t = u_0 - \int_0^t\mathcal{P}\mathcal{L}_{u_s}u_s\ ds - \nu\int_0^t A u_s\, ds + \frac{1}{2}\int_0^t\sum_{i=1}^\infty \mathcal{P}B_i^2u_s ds - \int_0^t \mathcal{P}Bu_s d\mathcal{W}_s. 
\end{equation}
     Then $u_{\cdot}\mathbbm{1}_{\cdot \leq \tau}$ is progressively measurable in $W^{m+1,2}_{\sigma}$ and such that for $\mathbbm{P}-a.e.$ $\omega$, $u_{\cdot}(\omega) \in C\left([0,T];W^{m,2}_{\sigma}\right)$ and $u_{\cdot}(\omega)\mathbbm{1}_{\cdot \leq \tau(\omega)} \in L^2\left([0,T];W^{m+1,2}_{\sigma}\right)$. Moreover $u$ satisfies
     \begin{equation} \nonumber
    u_t = u_0 - \int_0^{t \wedge \tau}\mathcal{P}\mathcal{L}_{u_s}u_s\ ds - \nu\int_0^{t \wedge \tau} A u_s\, ds  - \int_0^{t \wedge \tau} \mathcal{P}Bu_s \circ d\mathcal{W}_s 
\end{equation}
$\mathbbm{P}-a.s.$ in $L^2_{\sigma}$ for all $t \geq 0$. 
If $N=2$ then one can choose $\tau := T$.
\end{proposition}

We stress that the difference in this setting compared to the Navier boundary is that $\mathcal{P}\mathcal{L}: D(A^{\frac{m+1}{2}}) \rightarrow D(A^{\frac{m}{2}})$ for $m \geq 2$. At least for the Navier boundary we have that $\mathcal{P}\mathcal{L}: D(A) \rightarrow D(A^{\frac{1}{2}})$, which is untrue for the corresponding spaces of the no-slip condition, and represents the reason that the additional degree of regularity obtained in Theorem \ref{Stratotheorem2} is possible for this boundary condition but not for the no-slip condition.

\section{Appendix I: Weak and Strong Solutions to Nonlinear SPDEs} \label{section appendix I}

\subsection{Functional Framework}\label{subs functional framework}

This appendix is concerned with a variational framework for an abstract It\^{o} SPDE  \begin{equation} \label{thespde} \sy_t = \sy_0 + \int_0^t \mathcal{A}(s,\sy_s)ds + \int_0^t\mathcal{G} (s,\sy_s) d\mathcal{W}_s\end{equation}
which we pose for a triplet of embedded separable Hilbert Spaces $$V \hookrightarrow H \hookrightarrow U$$ whereby the embeddings are continuous linear injections, and $H \hookrightarrow U$ is compact. The equation (\ref{thespde}) is posed on a time interval $[0,T]$ for arbitrary but henceforth fixed $T \geq 0$. The mappings $\mathcal{A},\mathcal{G}$ are such that
    $\mathcal{A}:[0,T] \times V \rightarrow U,
    \mathcal{G}:[0,T] \times H \rightarrow \mathscr{L}^2(\mathfrak{U};U)$ are measurable. Understanding $\mathcal{G}$ as a mapping $\mathcal{G}: [0,T] \times H \times \mathfrak{U} \rightarrow U$, we introduce the notation $\mathcal{G}_i(\cdot,\cdot):= \mathcal{G}(\cdot,\cdot,e_i)$. We further impose the existence of a system of elements $(a_k)$ of $V$ which form an orthogonal basis of $U$ and a basis of $H$. Let us define the spaces $V_n:= \textnormal{span}\left\{a_1, \dots, a_n \right\}$ and $\mathcal{P}_n$ as the orthogonal projection to $V_n$ in $U$, that is $$\mathcal{P}_n:f \mapsto \sum_{k=1}^n\inner{f}{a_k}_Ua_k.$$
    It is required that the $(\mathcal{P}_n)$ are uniformly bounded in $H$, which is to say that there exists a constant $c$ independent of $n$ such that for all $f \in H$, \begin{equation} \label{uniform bounds of projection}
        \norm{\mathcal{P}_nf}_H \leq c\norm{f}_H.
    \end{equation}
    Moreover, our setup can be expanded by considering the induced Gelfand Triple $$H \xhookrightarrow{} U \xhookrightarrow{} H^*$$
defined relative to the inclusion mapping $i: H \rightarrow U$; indeed, the embedding of $U$ into $H^*$ is given by the composition of the isomorphism mapping $U$ into $U^*$ with the adjoint $i^*: U^* \rightarrow H^*$. In particular, the duality pairing between $H$ and $H^*$, $\inner{\cdot}{\cdot}_{H^* \times H}$, is compatible with $\inner{\cdot}{\cdot}_U$ in the sense that for for any $f \in U$, $g \in H$,
$$\inner{f}{g}_{H^* \times H} = \inner{f}{g}_U.$$
We assume that $\mathcal{A}:[0,T] \times H \rightarrow H^*$ is measurable. Specific bounds on the mappings $\mathcal{A}$ and $\mathcal{G}$ will be imposed in Assumption Sets 1, 2 and 3. We shall let $c_{\cdot}:[0,T]\rightarrow \R$ denote any bounded function, and for any constant $p \in \R$ we define the functions $K_U: U \rightarrow \R$, $K_H: H \rightarrow \R$, $K_V: V \rightarrow \R$ by
\begin{equation} \nonumber
    K_U(\phi)= 1 + \norm{\phi}_U^p, \quad K_H(\phi)= 1 + \norm{\phi}_H^p, \quad K_V(\phi)= 1 + \norm{\phi}_V^p.
\end{equation}
We may also consider these mappings as functions of two variables, e.g. $K_U: U \times U \rightarrow \R$ by $$K_U(\phi,\psi) = 1 + \norm{\phi}_U^p + \norm{\psi}_U^p.$$ Our assumptions will be stated for `the existence of a $K$ such that...' where we really mean `the existence of a $p$ such that, for the corresponding $K$, ...'.

\subsection{Assumption Set 1} \label{subby assumption}

 Recall the setup and notation of Subsection \ref{subs functional framework}. We assume that there exists a $c_{\cdot}$, $K$ and $\gamma > 0$ such that for all $\phi,\psi \in V$, $f \in H$ and $t \in [0,T]$:
 
 
  \begin{assumption} \label{new assumption 1} \begin{align}
     \label{111} \norm{\mathcal{A}(t,f)}_{H^*} +\sum_{i=1}^\infty \norm{\mathcal{G}_i(t,f)}^2_U &\leq c_t K_U(f)\left[1 + \norm{f}_H^2\right],\\ \label{222}
     \norm{\mathcal{A}(t,\phi) - \mathcal{A}(t,\psi)}_U^2 &\leq  c_tK_V\norm{\phi-\psi}_V^2,\\ \label{333}
    \sum_{i=1}^\infty \norm{\mathcal{G}_i(t,\phi) - \mathcal{G}_i(t,\psi)}_U^2 &\leq c_tK_V(\phi,\psi)\norm{\phi-\psi}_H^2.
 \end{align}
 \end{assumption}

\begin{assumption} \label{first assumption for uniform bounds}
    \begin{align}
   \label{uniformboundsassumpt1actual}  2\inner{\mathcal{A}(t,\phi)}{\phi}_U + \sum_{i=1}^\infty\norm{\mathcal{G}_i(t,\phi)}_U^2 &\leq c_t\left[1 + \norm{\phi}_U^2\right] - \gamma\norm{\phi}_H^2,\\  \label{uniformboundsassumpt2actual}
    \sum_{i=1}^\infty \inner{\mathcal{G}_i(t,\phi)}{\phi}^2_U &\leq c_t\left[1 + \norm{\phi}_U^4\right].
\end{align}
\end{assumption}

\begin{assumption}\footnote{In fact in (\ref{tightnessassumpt1}), the exponent $3/2$ could be replaced by any $q < 2$.}\label{tightness assumptions}
    \begin{align}
   \label{tightnessassumpt1}  \inner{\mathcal{A}(t,\phi)}{f}_U  &\leq c_t\left[K_U(\phi) + \norm{\phi}_H^{\frac{3}{2}} \right]\left[K_U(f) + \norm{f}_H^{\frac{3}{2}} \right],\\  \label{tightnessassumpt2}
    \sum_{i=1}^\infty \inner{\mathcal{G}_i(t,\phi)}{f}^2_U &\leq c_tK_U(\phi)K_H(f).
\end{align}
\end{assumption}

\begin{assumption}\label{limity assumptions}
    \begin{align}
   \label{limityassumpt1}  \inner{\mathcal{A}(t,\phi) - A(t,f)}{\psi}_{H^* \times H}  &\leq c_tK_V(\psi)\left[1 + \norm{\phi}_H + \norm{f}_H \right]\norm{\phi - f}_U,\\  \label{limityassumpt2}
    \sum_{i=1}^\infty \inner{\mathcal{G}_i(t,\phi) - \mathcal{G}_i(t,f)}{\psi}^2_U &\leq c_tK_V(\psi)\norm{\phi - f}_U^2.
\end{align}
\end{assumption}

\subsection{Assumption Set 2} \label{assumption set 2}

 Recall the setup and notation of Subsection \ref{subs functional framework}. We assume that there exists a $c_{\cdot}$, $K$ and $\gamma > 0$ such that for all $f,g \in H$ and $t \in [0,T]$:

  \begin{assumption} \label{new assumptionzizzle 1} \begin{align}
     \label{111zizzle} \norm{\mathcal{A}(t,f)}_{H^*}^2  &\leq c_t K_U(f)\left[1 + \norm{f}_H^2\right].
 \end{align}
 \end{assumption}

\begin{assumption} \label{therealcauchy assumptions}
\begin{align}
  \nonumber 2\inner{\mathcal{A}(t,f) - \mathcal{A}(t,g)}{f - g}_{H^* \times H} &+ \sum_{i=1}^\infty\norm{\mathcal{G}_i(t,f) - \mathcal{G}_i(t,g)}_U^2\\ \label{therealcauchy1} &\leq  c_{t}K_U(f,g)\left[1 + \norm{f}_H^2 + \norm{g}_H^2\right]\norm{f-g}_U^2 - \gamma\norm{f-g}_H^2,\\ \label{therealcauchy2}
    \sum_{i=1}^\infty \inner{\mathcal{G}_i(t,f) - \mathcal{G}_i(t,g)}{f-g}^2_U & \leq c_{t} K_U(f,g)\left[1 + \norm{f}_H^2 + \norm{g}_H^2\right] \norm{f-g}_U^4.
\end{align}
\end{assumption}

\subsection{Assumption Set 3} \label{assumption set 3}

Recall the setup and notation of Subsection \ref{subs functional framework}. We now impose the existence of a new Banach Space $\bar{H}$ which is an extension of $H$, or precisely, $H \subseteq \bar{H} \subseteq U$ and for every $f \in \bar{H}$, $\norm{f}_{\bar{H}} = \norm{f}_H.$ In addition, $\mathcal{G}:[0,T] \times V \rightarrow \mathscr{L}^2\left(\mathfrak{U};\bar{H}\right)$ is assumed measurable. We also suppose that there exists a real valued sequence $(\mu_n)$ with $\mu_n \rightarrow \infty$ such that for any $f \in \bar{H}$, \begin{align}
     \label{mu2}
    \norm{(I - \mathcal{P}_n)f}_U \leq \frac{1}{\mu_n}\norm{f}_{\bar{H}}
\end{align}
where $I$ represents the identity operator in $U$. Furthermore we assume that there exists a $\gamma > 0$ such that for any $\varepsilon > 0$, there exists a $c_\cdot$, $K$ (dependent on $\varepsilon$) such that for any $\phi \in V$, $\phi^n \in V_n$ and $t \in [0,T]$:

\begin{assumption} \label{newfacilitator}
\begin{align}
 \label{111fac} \norm{\mathcal{A}(t,\phi)}_{U}^2 + \sum_{i=1}^\infty\norm{\mathcal{G}_i(t,\phi)}_{\bar{H}}^2 \leq c_t K_U(\phi)\left[1 + \norm{\phi}_H^4 + \norm{\phi}_V^2\right]
    \end{align}
\end{assumption}

\begin{assumption} \label{assumptions for uniform bounds2}
 \begin{align}
   \label{uniformboundsassumpt1}  2\inner{\mathcal{P}_n\mathcal{A}(t,\phi^n)}{\phi^n}_H + \sum_{i=1}^\infty\norm{\mathcal{P}_n\mathcal{G}_i(t,\phi^n)}_H^2 &\leq c_tK_U(\phi^n)\left[1 + \norm{\phi^n}_H^4\right] - \gamma\norm{\phi^n}_V^2,\\  \label{uniformboundsassumpt2}
    \sum_{i=1}^\infty \inner{\mathcal{P}_n\mathcal{G}_i(t,\phi^n)}{\phi^n}^2_H &\leq c_tK_U(\phi^n)\left[1 + \norm{\phi^n}_H^6\right] + \varepsilon \norm{\phi^n}_V^2.
\end{align}
\end{assumption}

\subsection{Martingale Weak Solutions} \label{subby marty weak def and res}

We now state the definition and main result for martingale weak solutions.

\begin{definition} \label{definitionofspacetimeweakmartingale}
Let $\sy_0: \Omega \rightarrow U$ be $\mathcal{F}_0-$measurable. If there exists a filtered probability space $\left(\tilde{\Omega},\tilde{\mathcal{F}},(\tilde{\mathcal{F}}_t), \tilde{\mathbbm{P}}\right)$, a Cylindrical Brownian Motion $\tilde{\mathcal{W}}$ over $\mathfrak{U}$ with respect to $\left(\tilde{\Omega},\tilde{\mathcal{F}},(\tilde{\mathcal{F}}_t), \tilde{\mathbbm{P}}\right)$, an $\mathcal{F}_0-$measurable $\tilde{\sy}_0: \tilde{\Omega} \rightarrow U$ with the same law as $\sy_0$, and a progressively measurable process $\tilde{\sy}$ in $H$ such that for $\tilde{\mathbbm{P}}-a.e.$ $\tilde{\omega}$, $\tilde{\sy}_{\cdot}(\omega) \in C_w\left([0,T];U\right) \cap L^2\left([0,T];H\right)$\footnote{Note that $C_w\left([0,T];U\right) \subseteq L^{\infty}\left([0,T];U\right)$.} and
\begin{align} 
       \tilde{\sy}_t = \tilde{\sy}_0 + \int_0^t \mathcal{A}(s,\tilde{\sy}_s)ds + \int_0^t\mathcal{G} (s,\tilde{\sy}_s) d\mathcal{W}_s \label{newid1}
\end{align}
holds $\tilde{\mathbbm{P}}-a.s.$ in $H^*$ for all $t \in [0,T]$, then $\tilde{\sy}$ is said to be a martingale weak solution of the equation (\ref{thespde}).
\end{definition}


\begin{theorem} \label{theorem for martingale weak existence}
    Let Assumption Set 1 hold. For any given $\mathcal{F}_0-$measurable $\sy_0 \in L^\infty\left(\Omega;U\right)$, there exists a martingale weak solution of the equation (\ref{thespde}). 
\end{theorem}

\begin{proof}
    See [\cite{goodair2024weak}] Theorem 2.7.
\end{proof}


\subsection{Weak Solutions} \label{subbie weak def and res}

We now state the definitions and main result for weak solutions.

\begin{definition} \label{definitionofweak}
Let $\sy_0: \Omega \rightarrow U$ be $\mathcal{F}_0-$measurable. A process $\sy$ which is progressively measurable in $H$ and such that for $\mathbbm{P}-a.e.$ $\omega$, $\sy_{\cdot}(\omega) \in C\left([0,T];U\right) \cap L^2\left([0,T];H\right)$, is said to be a weak solution of the equation (\ref{thespde}) if the identity (\ref{thespde}) holds $\mathbbm{P}-a.s.$ in $H^*$ for all $t\in[0,T]$.
\end{definition}

\begin{definition} \label{definitionunique}
    A weak solution $\sy$ of the equation (\ref{thespde}) is said to be the unique solution if for any other such solution $\py$, $$ \mathbbm{P}\left(\left\{\omega \in \Omega: \sy_t(\omega) = \py_t(\omega) \quad \forall t \geq 0\right\}\right) = 1.$$
\end{definition}

\begin{theorem} \label{theorem for weak existence}
    Let Assumption Sets 1 and 2 hold. For any given $\mathcal{F}_0-$measurable $\sy_0: \Omega \rightarrow U$, there exists a unique weak solution of the equation (\ref{thespde}). 
\end{theorem}

\begin{proof}
    See [\cite{goodair2024weak}] Theorem 3.5.
\end{proof}


\subsection{Strong Solutions} \label{def and results strong sols}

We now state the definitions and main result for strong solutions.

\begin{definition} \label{definitionofstrong}
Let $\sy_0: \Omega \rightarrow H$ be $\mathcal{F}_0-$measurable. A process $\sy$ which is progressively measurable in $V$ and such that for $\mathbbm{P}-a.e.$ $\omega$, $\sy_{\cdot}(\omega) \in L^{\infty}\left([0,T];H\right) \cap L^2\left([0,T];V\right)$, is said to be a strong solution of the equation (\ref{thespde}) if the identity (\ref{thespde}) holds $\mathbbm{P}-a.s.$ in $U$ for all $t\in[0,T]$.
\end{definition}

Note that a strong solution necessarily has continuous paths in $U$, from the evolution equation satisfied in this space.

\begin{definition} \label{definitionunique2}
    A strong solution $\sy$ of the equation (\ref{thespde}) is said to be unique if for any other such solution $\py$, $$ \mathbbm{P}\left(\left\{\omega \in \Omega: \sy_t(\omega) = \py_t(\omega) \quad \forall t \geq 0\right\}\right) = 1.$$
\end{definition}

\begin{theorem}\label{theorem for strong existence}
    Let Assumption Sets 1, 2 and 3 hold. For any given $\mathcal{F}_0-$measurable $\sy_0:\Omega \rightarrow H$, there exists a unique strong solution of the equation (\ref{thespde}). 
\end{theorem}

\begin{proof}
    See [\cite{goodair2024weak}] Theorem 4.5.
\end{proof}

\begin{lemma} \label{continuity lemma 2}
    Let Assumption Sets 1, 2 and 3 hold. Suppose that there exists a continuous bilinear form $\inner{\cdot}{\cdot}_{U \times V}: U \times V \rightarrow \R$  such that for every $f \in H$, $\phi \in V$, $$ \inner{f}{\phi}_{U \times V} = \inner{f}{\phi}_H.$$ In addition, suppose that $\bar{H}$ can be taken as $H$. Then for $\mathbbm{P}-a.e.$ $\omega$, $\sy_{\cdot}(\omega) \in C\left([0,T];H\right)$. 
\end{lemma}

\begin{proof}
   See [\cite{goodair2024weak}] Lemma 4.14.
\end{proof}

\section{Appendix II: Maximal Solutions to Nonlinear SPDEs} \label{section maximal solution}

\subsection{Functional Framework} \label{sub functional for local}

Consider the same It\^{o} SPDE (\ref{thespde}), \begin{equation} \nonumber
    \sy_t = \sy_0 + \int_0^t \mathcal{A}(s,\sy_s)ds + \int_0^t\mathcal{G} (s,\sy_s) d\mathcal{W}_s
\end{equation}
which we pose for a triplet of embedded, separable Hilbert Spaces $$\mathscr{V} \hookrightarrow \mathscr{H} \hookrightarrow \mathscr{U}$$ whereby the embeddings are continuous linear injections. We ask that there is a continuous bilinear form $\inner{\cdot}{\cdot}_{\mathscr{U} \times \mathscr{V}}: \mathscr{U} \times \mathscr{V} \rightarrow \R$ such that for $f \in \mathscr{H}$ and $\psi \in \mathscr{V}$, \begin{equation}  \nonumber
    \inner{f}{\psi}_{\mathscr{U} \times \mathscr{V}} =  \inner{f}{\psi}_{\mathscr{H}}.
\end{equation}
The equation (\ref{thespde}) is posed on a time interval $[0,T]$ for arbitrary $T \geq 0$. The mappings $\mathcal{A},\mathcal{G}$ are such that
    $\mathcal{A}:[0,T] \times \mathscr{V} \rightarrow \mathscr{U},
    \mathcal{G}:[0,T] \times \mathscr{V} \rightarrow \mathscr{L}^2(\mathfrak{U};\mathscr{H})$ are measurable. Understanding $\mathcal{G}$ as a mapping $\mathcal{G}: [0,T] \times \mathscr{V} \times \mathfrak{U} \rightarrow \mathscr{H}$, we introduce the notation $\mathcal{G}_i(\cdot,\cdot):= \mathcal{G}(\cdot,\cdot,e_i)$. We further impose the existence of a system of elements $(a_k)$ of $\mathscr{V}$ with the following properties. Let us define the spaces $\mathscr{V}_n:= \textnormal{span}\left\{a_1, \dots, a_n \right\}$ and $\mathcal{P}_n$ as the orthogonal projection to $\mathscr{V}_n$ in $\mathscr{U}$. 
    It is required that the $(\mathcal{P}_n)$ are uniformly bounded in $H$, which is to say that there exists a constant $c$ independent of $n$ such that for all $\phi \in H$, \begin{equation}  \nonumber
        \norm{\mathcal{P}_nf}_\mathscr{H} \leq c\norm{f}_\mathscr{H}.
    \end{equation}
We also suppose that there exists a real valued sequence $(\mu_n)$ with $\mu_n \rightarrow \infty$ such that for any $f \in \mathscr{H}$, \begin{align}
      \nonumber
    \norm{(I - \mathcal{P}_n)f}_\mathscr{U} \leq \frac{1}{\mu_n}\norm{f}_{\mathscr{H}}
\end{align}
where $I$ represents the identity operator in $\mathscr{U}$. Specific bounds on the mappings $\mathcal{A}$ and $\mathcal{G}$ will be imposed in the following subsection. We shall use the same notation of $c,K$ from Subsection \ref{subs functional framework}.

\subsection{Assumptions} \label{assumptionschapter}

 We assume that there exists a $c_{\cdot}$, $K$ and $\gamma > 0$ such that for all $\phi,\psi \in \mathscr{V}$, $\phi^n \in \mathscr{V}_n$, $f \in H$ and $t \in [0,T]$:
 

  \begin{assumption}   \begin{align}
      \nonumber \norm{\mathcal{A}(t,\phi)}^2_\mathscr{U} +\sum_{i=1}^\infty \norm{\mathcal{G}_i(t,\phi)}^2_\mathscr{H} &\leq c_t K_\mathscr{U}(\phi)\left[1 + \norm{\phi}_\mathscr{V}^2\right],\\  \nonumber
     \norm{\mathcal{A}(t,\phi) - \mathcal{A}(t,\psi)}_\mathscr{U}^2 &\leq  c_tK_\mathscr{V}(\phi,\psi)\norm{\phi-\psi}_\mathscr{V}^2,\\  \nonumber
    \sum_{i=1}^\infty \norm{\mathcal{G}_i(t,\phi) - \mathcal{G}_i(t,\psi)}_\mathscr{U}^2 &\leq c_tK_\mathscr{U}(\phi,\psi)\norm{\phi-\psi}_\mathscr{H}^2.
 \end{align}
 \end{assumption}

\begin{assumption} \label{uniform assumpt}
 \begin{align}
    \nonumber  2\inner{\mathcal{P}_n\mathcal{A}(t,\phi^n)}{\phi^n}_\mathscr{H} + \sum_{i=1}^\infty\norm{\mathcal{P}_n\mathcal{G}_i(t,\phi^n)}_\mathscr{H}^2 &\leq c_tK_\mathscr{U}(\phi^n)\left[1 + \norm{\phi^n}_\mathscr{H}^4\right] - \gamma\norm{\phi^n}_\mathscr{V}^2,\\  \nonumber
    \sum_{i=1}^\infty \inner{\mathcal{P}_n\mathcal{G}_i(t,\phi^n)}{\phi^n}^2_\mathscr{H} &\leq c_tK_\mathscr{U}(\phi^n)\left[1 + \norm{\phi^n}_\mathscr{H}^6\right].
\end{align}
\end{assumption}

\begin{assumption} 
\begin{align}
  \nonumber 2\inner{\mathcal{A}(t,\phi) - \mathcal{A}(t,\psi)}{\phi - \psi}_\mathscr{U} &+ \sum_{i=1}^\infty\norm{\mathcal{G}_i(t,\phi) - \mathcal{G}_i(t,\psi)}_\mathscr{U}^2\\  \nonumber &\leq  c_{t}K_\mathscr{U}(\phi,\psi)\left[1 + \norm{\phi}_\mathscr{H}^2 + \norm{\psi}_\mathscr{H}^2\right] \norm{\phi-\psi}_\mathscr{U}^2 - \gamma\norm{\phi-\psi}_\mathscr{H}^2,\\  \nonumber
    \sum_{i=1}^\infty \inner{\mathcal{G}_i(t,\phi) - \mathcal{G}_i(t,\psi)}{\phi-\psi}^2_\mathscr{U} & \leq c_{t} K_\mathscr{U}(\phi,\psi)\left[1 + \norm{\phi}_\mathscr{H}^2 + \norm{\psi}_\mathscr{H}^2\right] \norm{\phi-\psi}_\mathscr{U}^4.
\end{align}
\end{assumption}

\begin{assumption} \label{my2.4}
\begin{align}
    \nonumber 2\inner{\mathcal{A}(t,\phi)}{\phi}_\mathscr{U} + \sum_{i=1}^\infty\norm{\mathcal{G}_i(t,\phi)}_\mathscr{U}^2 &\leq c_tK_\mathscr{U}(\phi)\left[1 +  \norm{\phi}_\mathscr{H}^2\right],\\ \nonumber
    \sum_{i=1}^\infty \inner{\mathcal{G}_i(t,\phi)}{\phi}^2_\mathscr{U} &\leq c_tK_\mathscr{U}(\phi)\left[1 + \norm{\phi}_\mathscr{H}^4\right].
\end{align}
\end{assumption}

\begin{assumption} 
 \begin{equation}  \nonumber
    \inner{\mathcal{A}(t,\phi)-\mathcal{A}(t,\psi)}{f}_\mathscr{U} \leq c_tK_\mathscr{U}(\phi,\psi)(1+\norm{f}_\mathscr{H})\left[1 + \norm{\phi}_\mathscr{V} + \norm{\psi}_\mathscr{V}\right]\norm{\phi-\psi}_\mathscr{H}.
    \end{equation}
\end{assumption}

\subsection{Definitions and Main Result} \label{subsection:notionsofsolution}

We state the definitions and main result.

\begin{definition} \label{v valued local def}
Let $\sy_0:\Omega \rightarrow \mathscr{H}$ be $\mathcal{F}_0-$ measurable. A pair $(\sy,\tau)$ where $\tau$ is a $\mathbbm{P}-a.s.$ positive stopping time and $\sy$ is a process such that for $\mathbbm{P}-a.e.$ $\omega$, $\sy_{\cdot}(\omega) \in C\left([0,T];\mathscr{H}\right)$ and $\sy_{\cdot}(\omega)\mathbbm{1}_{\cdot \leq \tau(\omega)} \in L^2\left([0,T];\mathscr{V}\right)$ for all $T \geq 0$ and with $\sy_{\cdot}\mathbbm{1}_{\cdot \leq \tau}$ progressively measurable in $\mathscr{V}$, is said to be a local strong solution of the equation (\ref{thespde}) if the identity
\begin{equation} \nonumber
    \sy_{t} = \sy_0 + \int_0^{t\wedge \tau} \mathcal{A}(s,\sy_s)ds + \int_0^{t \wedge \tau}\mathcal{G} (s,\sy_s) d\mathcal{W}_s
\end{equation}
holds $\mathbbm{P}-a.s.$ in $\mathscr{U}$ for all $t \geq 0$.
\end{definition}

\begin{definition} \label{V valued maximal definition}
A pair $(\sy,\Theta)$ such that there exists a sequence of stopping times $(\theta_j)$ which are $\mathbbm{P}-a.s.$ monotone increasing and convergent to $\Theta$, whereby $(\sy_{\cdot \wedge \theta_j},\theta_j)$ is a local strong solution of the equation (\ref{thespde}) for each $j$, is said to be a maximal strong solution of the equation (\ref{thespde}) if for any other pair $(\py,\Gamma)$ with this property then $\Theta \leq \Gamma$ $\mathbbm{P}-a.s.$ implies $\Theta = \Gamma$ $\mathbbm{P}-a.s.$.
\end{definition}

\begin{remark}
We do not require $\Theta$ to be finite in this definition, in which case we mean that the sequence $(\theta_j)$ is monotone increasing and unbounded for such $\omega$. 
\end{remark}

\begin{definition} \label{v valued maximal unique}
A maximal strong solution $(\sy,\Theta)$ of the equation (\ref{thespde}) is said to be unique if for any other such solution $(\py,\Gamma)$, then $\Theta = \Gamma$ $\mathbbm{P}-a.s.$ and \begin{equation} \nonumber\mathbbm{P}\left(\left\{\omega \in \Omega: \sy_{t}(\omega) =  \py_{t}(\omega)  \quad \forall t \in [0,\Theta) \right\} \right) = 1. \end{equation}
\end{definition}

\begin{theorem} \label{theorem1}
For any given $\mathcal{F}_0-$ measurable $\sy_0:\Omega \rightarrow \mathscr{H}$, there exists a unique maximal strong solution $(\sy,\Theta)$ of the equation (\ref{thespde}). Moreover at $\mathbbm{P}-a.e.$ $\omega$ for which $\Theta(\omega)<\infty$, we have that \begin{equation} \nonumber \sup_{r \in [0,\Theta(\omega))}\norm{\sy_r(\omega)}_\mathscr{U}^2 + \int_0^{\Theta(\omega)}\norm{\sy_r(\omega)}_\mathscr{H}^2dr = \infty\end{equation}
and in consequence for any $\mathbbm{P}-a.s.$ positive stopping time $\tau$ such that $$\sup_{r \in [0,\tau(\omega))}\norm{\sy_r(\omega)}_\mathscr{U}^2 + \int_0^{\tau(\omega)}\norm{\sy_r(\omega)}_\mathscr{H}^2dr < \infty $$
$\mathbbm{P}-a.s.$, $(\sy_{\cdot \wedge \tau}, \tau)$ is a local strong solution of the equation (\ref{thespde}).
\end{theorem}

\begin{proof}
    See [\cite{goodair2024improved}] Theorem 2.9.
\end{proof}

\section{Appendix III: Infinite Dimensional It\^{o}-Stratonovich Conversion} \label{Appendix II}

This theory is taken from [\cite{goodair2022stochastic}] Subsections 2.2 and 2.3, and is provided here for simplicity to apply in Subsection \ref{subs strong strat}. We work with a quartet of embedded Hilbert Spaces $$V \hookrightarrow H \hookrightarrow U \hookrightarrow X$$ where the embeddings are assumed to be continuous linear injections. We start from an SPDE \begin{equation} \label{stratoSPDE}
    \sy_t = \sy_0 + \int_0^t \mathcal{Q}\sy_sds + \int_0^t\mathcal{G}\sy_s \circ d\mathcal{W}_s.
\end{equation}
where the mappings $\mathcal{Q}$, $\mathcal{G}$ satisfy the following conditions, with the general operator $\tilde{K}:H \rightarrow \R$ defined by $$\tilde{K}(\phi):= c\left(1 + \norm{\phi}_U^p + \norm{\phi}_H^q\right)$$ for any constants $c,p,q$ independent of $\phi$.
 \begin{assumption} \label{Qassumpt}
$\mathcal{Q}: V \rightarrow U$ is measurable and for any $\phi \in V$, $$\norm{\mathcal{Q}\phi}_U \leq \tilde{K}(\phi)[1 + \norm{\phi}_V^2].$$
 \end{assumption}
 
\begin{assumption} \label{Gassumpt}
$\mathcal{G}$ is understood as a measurable mapping \begin{align*}
    \mathcal{G}: V \rightarrow \mathscr{L}^2(\mathfrak{U};H), \qquad
     \mathcal{G}: H \rightarrow \mathscr{L}^2(\mathfrak{U};U), \qquad
       \mathcal{G}: U \rightarrow \mathscr{L}^2(\mathfrak{U};X)
\end{align*}
defined over $\mathfrak{U}$ by its action on the basis vectors $\mathcal{G}(\cdot, e_i):= \mathcal{G}_i(\cdot).$ In addition each $\mathcal{G}_i$ is linear and there exists constants $c_i$ such that for all $\phi \in V$, $\psi \in H$, $\eta \in U$:
\begin{align*}
    \norm{\mathcal{G}_i\phi}_{H} \leq c_i \norm{\phi}_V, \qquad
    \norm{\mathcal{G}_i\psi}_{U} \leq c_i \norm{\psi}_H, \qquad
    \norm{\mathcal{G}_i\eta}_{X} \leq c_i \norm{\eta}_U, \qquad
    \sum_{i=1}^\infty c_i^2 < \infty.
\end{align*}
\end{assumption}

In this setting, we have the following result ([\cite{goodair2022stochastic}] Theorem 2.3.1).

\begin{theorem} \label{theorem for ito strat conversion}
    Suppose that $\sy$ is a process whereby for $\mathbbm{P}-a.e.$ $\omega$, $\sy_{\cdot}(\omega) \in C\left([0,T];H\right)\cap L^2\left([0,T];V\right)$, $\sy$ is progressively measurable in $V$, and moreover satisfies the identity
\begin{equation} \nonumber
    \sy_{t} = \sy_0 + \int_0^{t} \left(\mathcal{Q} + \frac{1}{2}\sum_{i=1}^\infty \mathcal{G}_i^2\right)\sy_sds + \int_0^{t }\mathcal{G}\sy_s d\mathcal{W}_s
\end{equation}
$\mathbbm{P}-a.s.$ in $U$ for all $t \in [0,T]$. Then $\sy$ satisfies the identity $$\sy_{t} = \sy_0 + \int_0^{t} \mathcal{Q}\sy_sds + \int_0^{t}\mathcal{G}\sy_s \circ d\mathcal{W}_s $$
$\mathbbm{P}-a.s.$ in $X$ for all $t \geq 0$.
\end{theorem}

The mapping $\frac{1}{2}\sum_{i=1}^\infty \mathcal{G}_i^2$ is understood as a pointwise limit, which is justified in [\cite{goodair2022stochastic}] Subsection 2.3.

\section{Appendix IV: Explosion of Galerkin Projections} \label{last appendix}





To illustrate a difficulty of high order estimates on the bounded domain, we prove the explosion of the Galerkin Projections when not restricted to a space satisfying the boundary condition. This was mentioned in the introduction. The result holds true for the no-slip condition as well, even including the case $k = 1$. 

\begin{lemma} \label{lemma for explosion of galerkin}
    Fix $k \in \N$ with $k \geq 2$. For every constant $c$, there exists an $n \in N$ and an $f \in W^{k,2} \cap L^2_{\sigma}$ such that $$ \norm{\bar{\mathcal{P}}_nf}_{W^{k,2}} > c\norm{f}_{W^{k,2}}.$$
\end{lemma}

\begin{proof}
    Take any $f \in W^{k,2} \cap L^2_{\sigma}$ which is not in $\bar{W}^{2,2}_{\alpha}$, in particular it does not satisfy the Navier boundary conditions. We assume for a contradiction that there exists a $c$ for which for all $n \in \N$, $$ \norm{\bar{\mathcal{P}}_nf}_{W^{k,2}} \leq c\norm{f}_{W^{k,2}}.$$
    In particular $$ \norm{\bar{\mathcal{P}}_nf}_{W^{2,2}} \leq c\norm{f}_{W^{k,2}}$$ so the sequence $(\bar{\mathcal{P}}_nf)$ is uniformly bounded in $\bar{W}^{2,2}_{\alpha}$ and thus admits a weakly convergent subsequence in $\bar{W}^{2,2}_{\alpha}$. This subsequence is also weakly convergent in $L^2_{\sigma}$ to the same limit, due to the embedding of $\bar{W}^{2,2}_{\alpha}$ into $L^2_{\sigma}$. However as $f \in L^2_{\sigma}$ then $(\bar{\mathcal{P}}_nf)$ converges to $f$ strongly in $L^2_{\sigma}$ hence weakly as well, so by uniqueness of limits in the weak topology then $f \in \bar{W}^{2,2}_{\alpha}$ which is a contradiction.
\end{proof}

\addcontentsline{toc}{section}{References}
\bibliographystyle{newthing}
\bibliography{myBibBib}

\end{document}